\documentclass[12pt,reqno]{amsart}
\usepackage[a4paper, hmargin={2.7cm,2.7cm},vmargin={3.3cm,3.3cm}]{geometry}
\usepackage{latexsym,amsmath,amssymb,amscd}
\usepackage[noadjust]{cite}
\usepackage{amsfonts}
\usepackage{amsthm}
\usepackage{enumerate}
\usepackage{amsthm}
\usepackage{todonotes}
\usepackage{cleveref}
\usepackage[all]{xy}
\usepackage{dsfont}
\usepackage{etoolbox}
\numberwithin{equation}{section}

\makeatletter
\patchcmd{\ttlh@hang}{\parindent\z@}{\parindent\z@\leavevmode}{}{}
\patchcmd{\ttlh@hang}{\noindent}{}{}{}
\makeatother

\newcommand\numberthis{\addtocounter{equation}{1}\tag{\theequation}}

\makeatletter
\providecommand*{\cupdot}{%
  \mathbin{%
    \mathpalette\@cupdot{}%
  }%
}
\newcommand*{\@cupdot}[2]{%
  \ooalign{%
    $\m@th#1\cup$\cr
    \sbox0{$#1\cup$}%
    \dimen@=\ht0 %
    \sbox0{$\m@th#1\cdot$}%
    \advance\dimen@ by -\ht0 %
    \dimen@=.5\dimen@
    \hidewidth\raise\dimen@\box0\hidewidth
  }%
}

\providecommand*{\bigcupdot}{%
  \mathop{%
    \vphantom{\bigcup}%
    \mathpalette\@bigcupdot{}%
  }%
}
\newcommand*{\@bigcupdot}[2]{%
  \ooalign{%
    $\m@th#1\bigcup$\cr
    \sbox0{$#1\bigcup$}%
    \dimen@=\ht0 %
    \advance\dimen@ by -\dp0 %
    \sbox0{\scalebox{2}{$\m@th#1\cdot$}}%
    \advance\dimen@ by -\ht0 %
    \dimen@=.5\dimen@
    \hidewidth\raise\dimen@\box0\hidewidth
  }%
}
\makeatother

\newtheorem{theorem}{Theorem}[section]
\newtheorem{lemma}[theorem]{Lemma}
\newtheorem{proposition}[theorem]{Proposition}
\newtheorem{corollary}[theorem]{Corollary}

\theoremstyle{definition}
\newtheorem{definition}[theorem]{Definition}

\theoremstyle{remark}

\newcommand{\lbdf}{D^-_0}
\newcommand{\ubdf}{D^+_0}
\DeclareMathOperator{\PW}{PW}

\newcommand{\Hpi}{\mathcal{H}_{\pi}}

\newcommand{\R}{\mathbb{R}}

\newcommand{\N}{\mathbb{N}}

\newcommand{\ve}{\varepsilon}

\newcommand{\lan}{\langle}
\newcommand{\ran}{\rangle}

\DeclareMathOperator{\diam}{diam}

\newcommand{\gf}{g_{\lambda}}
\newcommand{\df}{h_{\lambda}}

\title{Frame redundancy and Beurling density}

\author{Marcin Bownik}
\address{Department of Mathematics, University of Oregon, Eugene, OR 97403--1222, USA}
\email{mbownik@uoregon.edu}

\author{Jordy Timo van Velthoven}
\address{Faculty of Mathematics,
University of Vienna, 
Oskar-Morgenstern-Platz 1,
1090 Vienna, Austria}
\email{jordy-timo.van-velthoven@univie.ac.at}

\makeatletter
\@namedef{subjclassname@2020}{\textup{2020} Mathematics Subject Classification}
\makeatother



\begin{document}

\maketitle

\begin{abstract}
We show that the frame measure function of a frame in certain reproducing kernel Hilbert spaces on metric measure spaces is given by the reciprocal of the Beurling density of its index set. In addition, we show that each such frame with Beurling density greater than one contains a subframe with Beurling density arbitrary close to one. This confirms that the concept of frame measure function as introduced by Balan and Landau is a meaningful quantitative definition for the redundancy of a large class of infinite frames. In addition, it shows that the necessary density conditions for sampling in reproducing kernel Hilbert spaces obtained by F\"uhr, Gr\"ochenig, Haimi, Klotz and Romero are optimal. As an application, we also settle the open questions of the existence of frames near the critical density for exponential frames on unbounded sets and for nonlocalized Gabor frames. The techniques used in this paper combine a selector form of Weaver's conjecture  and various methods for quantifying the overcompleteness of frames.
\end{abstract}

\section{Introduction}
Let $\mathcal{H} \subseteq L^2 (\mathbb{R}^n)$ be a reproducing kernel Hilbert space, i.e., a closed subspace in which point evaluation is continuous, and let $\{k_x \}_{x \in \mathbb{R}^n}$ be the reproducing kernels in $\mathcal{H}$, so that $f(x) = \langle f, k_x \rangle$ for all $f \in \mathcal{H}$. A set $\Lambda \subseteq \mathbb{R}^n$ is called a \emph{sampling set} for $\mathcal{H} $ if the system $\{k_{\lambda} \}_{\lambda \in \Lambda}$ is a frame for $\mathcal{H}$, i.e., there exist $0 < A \leq B < \infty$ such that
\[
 A \| f \|^2 \leq \sum_{\lambda \in \Lambda} |\langle f, k_{\lambda} \rangle |^2 \leq B \| f \|^2 \quad \text{for all} \quad f \in \mathcal{H},
\]
and $\Lambda \subseteq \mathbb{R}^n$ is called an \emph{interpolation set} for $\mathcal{H}$ if $\{k_{\lambda} \}_{\lambda \in \Lambda}$ is a Riesz sequence in $\mathcal{H}$, i.e., there exist
$0 < A \leq B < \infty$ such that
\[
 A \| c \|^2 \leq \bigg\| \sum_{\lambda \in \Lambda} c_{\lambda} k_{\lambda} \bigg\|^2 \leq B \| c \|^2 \quad \text{for all} \quad c \in \ell^2(\Lambda).
\]
A set $\Lambda \subseteq \mathbb{R}^n$ is simultaneously a sampling set and an interpolation set for $\mathcal{H}$ if and only if the system of reproducing kernels $\{k_{\lambda} \}_{\lambda \in \Lambda}$ is a Riesz basis for $\mathcal{H}$.

The study of the existence of sampling sets and interpolation sets goes back to at least the fundamental work of Beurling \cite{beurling1966local, beurling1989collected} and Landau \cite{landau1967necessary} on classical Paley-Wiener spaces. In particular, the papers \cite{beurling1966local, beurling1989collected, landau1967necessary} provide necessary conditions for sampling sets and interpolation sets in terms of a notion of density of the set and identified a critical density that separates sampling sets from interpolation sets.
Since then, obtaining necessary density conditions for sampling sets and interpolation sets has been a topic of ongoing interest in various areas of mathematical analysis, such as complex analysis and harmonic analysis, see, e.g., \cite{lindholm2001sampling, marco2003interpolating, grochenig2019strict, seip2004interpolation, lev2016equidistribution} and \cite{grochenig2024necessary, balan2006density, enstad2025dynamical, olevskii2016functions, fuehr2017density} (and the references therein)
for a small selection of the relevant literature.

In the present paper, the existence of sampling sets with close to critical density will be shown in general reproducing kernel Hilbert spaces on metric measure spaces under common conditions on the metric measure space and the reproducing kernel. This allows us to answer various open questions on the existence of optimal sampling sets in particular settings in a unified manner.
Before stating the main results and describing the methods used in this paper, the essential background and context will be provided.

\subsection{Necessary density conditions}
We start by recalling the known necessary density conditions for general reproducing kernel Hilbert spaces as studied in \cite{fuehr2017density, mitkovsi2020density}. 
For simplicity, throughout this introduction, we restrict the discussion to reproducing kernel Hilbert spaces on Euclidean space equipped with Lebesgue measure and Euclidean metric. This setting covers already most of our key examples.

 Following \cite{fuehr2017density, mitkovsi2020density}, we make the following assumptions on the reproducing kernels $\{k_x\}_{x \in \mathbb{R}^n}$ of a reproducing kernel Hilbert space $\mathcal{H} \subseteq L^2 (\mathbb{R}^n)$. 
 First, we assume  $\{k_x\}_{x \in \mathbb{R}^n}$ satisfies the \emph{diagonal condition} (DC): There exist constants $0 < C_1 \leq C_2 < \infty$ such that
\begin{align} \label{eq:dc_intro} \tag{DC}
 C_1 \leq \| k_x \|^2 \leq C_2 \quad \text{for all} \quad x \in \mathbb{R}^n.
\end{align}
In addition, we assume that the system $\{k_x\}_{x \in \mathbb{R}^n}$ satisfies the \emph{weak localization} property (WL): For every $\varepsilon > 0$, there exists $r = r(\varepsilon) > 0$ such that
\begin{align} \label{eq:wl_intro} \tag{WL}
 \sup_{x \in \mathbb{R}^d} \int_{\mathbb{R}^n \setminus B_r (x)} |\langle k_x, k_y \rangle |^2 \; dy < \varepsilon^2.
\end{align}
Lastly, we will assume a discrete variant of the weak localization property, known as the \emph{homogeneous approximation property} (HAP): If $\Lambda \subseteq \mathbb{R}^n$ is a countable set such that $\{k_{\lambda} \}_{\lambda \in \Lambda}$ is a Bessel sequence in $\mathcal{H} \subseteq L^2 (\mathbb{R}^n)$ and $\varepsilon > 0$, then there exists $r = r(\varepsilon) > 0$ such that
\begin{align} \label{eq:hap_intro} \tag{HAP}
 \sup_{x \in \mathbb{R}^d} \sum_{\lambda \in \Lambda \setminus B_r (x)} |\langle k_x, k_{\lambda} \rangle |^2  < \varepsilon^2.
\end{align}
Under these assumptions, the following necessary density conditions for sampling sets and interpolation sets were shown in \cite[Corollary 4.1]{fuehr2017density}.

\begin{theorem}[\cite{fuehr2017density}] \label{thm:necessary_intro}
Let $\mathcal{H} \subseteq L^2 (\mathbb{R}^n)$ be a reproducing kernel Hilbert space with reproducing kernels $\{k_x \}_{x \in \mathbb{R}^n}$. Suppose that $\{k_x \}_{x \in \mathbb{R}^n}$ satisfies the diagonal condition \eqref{eq:dc_intro}, the weak localization \eqref{eq:wl_intro} and the homogeneous approximation property \eqref{eq:hap_intro}. Then the following assertions hold:
\begin{enumerate}
 \item[(i)] If $\{k_{\lambda} \}_{\lambda \in \Lambda}$ is a frame for $\mathcal{H}$, then
 \[
  D^-_0 (\Lambda) := \liminf_{r \to \infty} \inf_{x \in \mathbb{R}^n} \frac{\#\big(\Lambda \cap B_r (x) \big)}{\int_{B_r(x)} k(y,y) \; dy} \geq 1.
 \]

\item[(ii)] If $\{k_{\lambda} \}_{\lambda \in \Lambda}$ is a Riesz sequence in $\mathcal{H}$, then
 \[
  D^+_0 (\Lambda) := \limsup_{r \to \infty} \sup_{x \in \mathbb{R}^n} \frac{\#\big(\Lambda \cap B_r (x) \big)}{\int_{B_r(x)} k(y,y) \; dy} \leq 1.
 \]
\end{enumerate}
In particular, if $\{k_{\lambda} \}_{\lambda \in \Lambda}$ is a Riesz basis for $\mathcal{H}$, then $D_0^+(\Lambda) = D^-_0 (\Lambda) = 1$.
\end{theorem}

The densities $D^{-}_0 (\Lambda)$ and $D^+_0(\Lambda)$ of a set $\Lambda \subseteq \mathbb{R}^n$ appearing in Theorem \ref{thm:necessary_intro} are sometimes referred to as the \emph{dimension-free versions} of the \emph{lower Beurling density} and \emph{upper Beurling density}, and they can be interpreted as Beurling densities with respect to the weighted measure $d\mu(y) = k(y,y) dy$. For reproducing kernels satisfying the diagonal condition \eqref{eq:dc} with $C_1 = C_2$, the densities $D_0^{\pm}$ coincide (up to a scalar multiple) with the standard Beurling densities $D^{\pm}$ with respect to the Lebesgue measure. For reproducing kernels whose norms are not constant, however, the dimension-free densities appear most appropriate; see also \cite[Remark 2.2]{grochenig2019strict}.

\subsection{Frames near the critical density} \label{sec:frame_intro}
Theorem \ref{thm:necessary_intro} provides a critical density that separates sampling sets from interpolation sets. It is therefore meaningful to aim for establishing optimal sampling sets, in the sense that their density is arbitrary close to the critical density $1$. So far, this has only been  established for particular examples, such as exponential frames on compact sets \cite{bownik2025on, agora2015multi, marzo2006riesz}, localized Gabor frames \cite{balan2011redundancy} and weighted Fock spaces of entire functions \cite{grochenig2019strict}; see also \cite{grochenig2018sampling, below2023gabor} for various sufficient density conditions for Gabor frames and \cite{olevskii2008universal, matei2009variant} for universal sampling sets for exponential systems. It is worth mentioning
that these known results on optimal sampling sets are obtained using arguments that are particular for the underlying setting, and hence appear to not generalize to the setting of Theorem \ref{thm:necessary_intro}.
As a matter of fact, the existence of optimal sampling sets is mentioned as an open question for spectral subspaces of elliptic differential operators in \cite[p. 590]{grochenig2024necessary} and for coherent state subsystems arising from nilpotent Lie groups as \cite[Question 11]{velthoven2024density}.

Our first main result shows the existence of optimal sampling sets in the general setting of Theorem \ref{thm:necessary_intro}. The precise statement is as follows:

\begin{theorem} \label{thm:main_intro}
 Let $\mathcal{H} \subseteq L^2 (\mathbb{R}^n)$ be a reproducing kernel Hilbert space with reproducing kernels $\{k_x \}_{x \in \mathbb{R}^n}$. Suppose that $\{k_x \}_{x \in \mathbb{R}^n}$ satisfies the diagonal condition \eqref{eq:dc_intro}, the weak localization \eqref{eq:wl_intro} and the homogeneous approximation property \eqref{eq:hap_intro}.

 Then, for every $\varepsilon > 0$, there exists $\Lambda \subseteq \mathbb{R}^n$ satisfying
 \[
  D^+_0 (\Lambda) \leq 1 + \varepsilon
 \]
 and such that $\{k_{\lambda} \}_{\lambda \in \Lambda}$ is a frame for $\mathcal{H}$.
\end{theorem}

Theorem \ref{thm:main_intro} is a special case of a more general theorem (\Cref{thm:main}) on metric measure spaces. These results solve, in particular, the aforementioned questions mentioned in \cite{grochenig2024necessary, velthoven2024density} on the existence of frames near the critical density for spectral subspaces of elliptic differential operators (cf. \Cref{sec:elliptic}) and coherent frames arising from nilpotent Lie groups (cf. \Cref{sec:coherent}). In addition, Theorem \ref{thm:main_intro} recovers the known results on frames near the critical density for exponential frames on compact sets \cite{bownik2025on, agora2015multi, marzo2006riesz} (cf. \Cref{sec:exponential}), localized Gabor frames \cite{balan2011redundancy} (cf. \Cref{sec:gabor}) and weighted Fock spaces of entire functions \cite{grochenig2019strict} (cf. \Cref{sec:holomorphic}) in a unified manner.

In addition to \Cref{thm:main_intro}, we also show the existence of frames near the critical density in two settings in which the homogeneous approximation property \eqref{eq:hap_intro} seems not applicable directly in general. For both settings, we can phrase the result simply using the upper Beurling density with respect to Lebesgue measure:
\[
D^+ (\Lambda) := \limsup_{r \to \infty} \sup_{x \in \mathbb{R}^n} \frac{\#(\Lambda \cap B_r (x))}{|B_r(x)|},
\]
where we denote by $|U|$ the Lebesgue measure of a measurable set $U \subseteq \mathbb{R}^n$.

The first of these settings is that of exponential frames on possibly unbounded sets:

\begin{theorem} \label{thm:exponential_intro}
 Let $\Omega \subseteq \mathbb{R}^d$ be a set of finite measure. For every $\varepsilon > 0$, there exists $\Lambda \subseteq \mathbb{R}^d$ satisfying
 \[
  D^+ (\Lambda)  \leq (1+\varepsilon) |\Omega|
 \]
and such that $\{e^{2\pi i \lambda \cdot} \mathds{1}_{\Omega} \}_{\lambda \in \Lambda}$ is a frame for $L^2 (\Omega)$.
\end{theorem}

The existence of exponential frames on possibly unbounded sets was shown in \cite{nitzan2016exponential}; see also \cite{freeman2019discretization}. Theorem \ref{thm:exponential} improves this existence result by showing that such frames can even be chosen near the critical density. Although for certain sets $\Omega \subseteq \mathbb{R}^d$, it is known that Riesz bases of exponentials exist \cite{kozma2015combining, kozma2016combining, debernardi2022riesz}, not every set of finite measure $\Omega \subseteq \mathbb{R}$ admits a Riesz basis \cite{kozma2023set}. The recent paper \cite{enstad2025exponential} shows that such a set does not even admit a frame at the critical density $D^-(\Lambda) = |\Omega|$, which shows that \Cref{thm:exponential_intro} is optimal.

In addition to the mere existence of exponential frames near the critical density as asserted in \Cref{thm:exponential_intro}, we also show that such frames can be chosen with ``good'' frame bounds, which extends the main result of \cite{bownik2025on} from compact sets to general sets of finite measure. See \Cref{thm:exponential} for a precise statement.

Lastly, we state the following theorem on Gabor frames, which does not require any localization assumption on the generating window function.

\begin{theorem} \label{thm:gabor_intro}
Let $g \in L^2 (\mathbb{R}^d)$ be nonzero. For every $\varepsilon > 0$, there exists  $\Lambda \subseteq \mathbb{R}^{d} \times \mathbb{R}^d$ satisfying
\[
 D^+ (\Lambda) \leq 1 + \varepsilon
\]
and such that $\big\{ e^{2\pi i \lambda_1 \cdot} g(\cdot - \lambda_2) \big\}_{(\lambda_1, \lambda_2) \in \Lambda}$ is a frame for $L^2 (\mathbb{R}^d)$.
\end{theorem}

In \cite[Theorem 1.2]{balan2011redundancy}, the existence of Gabor frames near the critical density was shown under the assumption that the generating window belongs to the modulation space $M^1 (\mathbb{R}^d)$. Theorem \ref{thm:gabor_intro} extends this result to arbitrary nonzero functions in $L^2 (\mathbb{R}^d)$. It is worth mentioning that for a general nonzero  $g \in L^2 (\mathbb{R}^d)$, there might not exist $\Lambda \subseteq \mathbb{R}^{d} \times \mathbb{R}^d$ with $D^-(\Lambda) = 1$ that yields a Gabor frame $\big\{ e^{2\pi i \lambda_1 \cdot} g(\cdot - \lambda_2) \big\}_{(\lambda_1, \lambda_2) \in \Lambda}$
for $L^2 (\mathbb{R}^d)$. Indeed, the Balian-Low theorem  asserts that for a frame $\big\{ e^{2\pi i \lambda_1 \cdot} g(\cdot - \lambda_2) \big\}_{(\lambda_1, \lambda_2) \in \Lambda}$ with Schwarz function $g \in \mathcal{S}(\mathbb{R}^d)$, the necessary density condition $D^- (\Lambda) \geq 1$ must be a strict inequality, see, e.g., \cite[Theorem 1.5]{ascensi2014dilation}. Hence, \Cref{thm:gabor_intro} is optimal, in the sense that frames at the critical density can in general not be expected.

\subsection{Frame redundancy}
With notation as in \Cref{thm:necessary_intro} and \Cref{thm:main_intro}, a key ingredient in our proof of \Cref{thm:main_intro} is formed by the following identities involving a frame $\{k_{\lambda} \}_{\lambda \in \Lambda}$ and its canonical dual frame $\{h_{\lambda} \}_{\lambda \in \Lambda}$:
\begin{align} \label{eq:frd_intro1}
M^- (\{k_{\lambda} \}_{\lambda \in \Lambda}) := \liminf_{r \to \infty} \inf_{x \in \mathbb{R}^n} \frac{1}{\#(\Lambda \cap B_r (x))} \sum_{\lambda \in \Lambda \cap B_r(x)} \langle k_{\lambda} , h_{\lambda} \rangle = \frac{1}{D^+_0 (\Lambda)}
\end{align}
and \begin{align} \label{eq:frd_intro2}
M^+ (\{k_{\lambda} \}_{\lambda \in \Lambda}) := \limsup_{r \to \infty} \sup_{x \in \mathbb{R}^n} \frac{1}{\#(\Lambda \cap B_r (x))} \sum_{\lambda \in \Lambda \cap B_r(x)} \langle k_{\lambda} , h_{\lambda} \rangle = \frac{1}{D^-_0 (\Lambda)}.
\end{align}
The quantities $0 \leq M^{\pm} (\{k_{\lambda} \}_{\lambda \in \Lambda}) \leq 1$ are special cases of the so-called (ultra-filter) \emph{frame measure function} of a frame that is studied in detail in \cite{balan2007measure} with the aim of quanitifying the redundancy of an infinite frame; see also \cite{balan2006density, balan2011redundancy}. The identities \eqref{eq:frd_intro1} and \eqref{eq:frd_intro2} relating the frame measure function and Beurling density have been proven for Gabor frames in \cite{balan2006density2} and for more general coherent frames in \cite{caspers2023overcompleteness}. In \Cref{thm:framemeasure_rkhs} we prove the identities \eqref{eq:frd_intro1} and \eqref{eq:frd_intro2} under the same assumptions as in Theorems \ref{thm:necessary_intro} and \ref{thm:main_intro}. We show this through a variation on the technique that is used in \cite{fuehr2017density, mitkovsi2020density} for proving the necessary density conditions for sampling and interpolation sets. Note that, in particular, the identity \eqref{eq:frd_intro2} implies the sampling part (i) of \Cref{thm:necessary_intro}.

In \cite{balan2007measure}, it is proposed that the reciprocal of the frame measure function $M^{\pm}$ be the redundancy for an infinite frame and is shown to be a quantitatively meaningful notion for the definition of the redundancy of an infinite frame, cf. \cite[Section 10]{balan2007measure} (see also \cite[p. 110]{balan2006density}).
For example, with this definition, any frame has redundancy greater than or equal to $1$, any Riesz basis has redundancy precisely $1$, and the redundancy is additive on finite unions of frames.
Nevertheless, an important question (cf. \cite[p. 673]{balan2007measure}) that is unanswered in general is whether any frame with redundancy greater than $1$ contains a subframe that is still a frame for the same space with redundancy $1$ or $1+\varepsilon$ for any $\varepsilon > 0$. For localized Gabor frames, where the redundancy coincides with the Beurling density of its index set, this question has been solved in \cite[Theorem 1.2]{balan2011redundancy} by showing that any such Gabor frame with density greater than $1$ admits a subsystem that still forms a Gabor frame and has density smaller than $1+\varepsilon$ for arbitrary $\varepsilon > 0$. We extend this result to general Gabor frames with an arbitrary window $g\in L^2(\R^d)$.

In fact, we show that all of the frames discussed in \Cref{sec:frame_intro} admit subframes with density arbitrary close to the critical density. More precisely, we prove the following theorem:

\begin{theorem} \label{thm:redundancy_intro}
Let $\mathcal{H} \subseteq L^2 (\mathbb{R}^n)$ be a reproducing kernel Hilbert space with reproducing kernels $\{k_x \}_{x \in \mathbb{R}^n}$. Suppose that $\{k_x \}_{x \in \mathbb{R}^n}$ satisfies the diagonal condition \eqref{eq:dc_intro} and that for any countable set $\Lambda \subseteq \mathbb{R}^n$ such that $\{ k_{\lambda} \}_{\lambda \in \Lambda}$ is a frame for $\mathcal{H}$, the following formulae hold:
\begin{align} \label{eq:frd_intro}
 M^- (\{k_{\lambda} \}_{\lambda \in \Lambda} ) = \frac{1}{D^+_0 (\Lambda)} \quad \text{and} \quad M^+ (\{k_{\lambda} \}_{\lambda \in \Lambda} ) = \frac{1}{D^-_0 (\Lambda)}.
\end{align}
Then, for every $\varepsilon > 0$, there exists $\Gamma \subseteq \Lambda$ satisfying $D^-_0(\Gamma) \leq 1+ \varepsilon$ and such that $\{ k_{\gamma} \}_{\gamma \in \Gamma}$ is a frame for $\mathcal{H}$.
\end{theorem}

As mentioned above, the identities in \eqref{eq:frd_intro} are satisfied whenever the reproducing kernels $\{k_x \}_{x \in \mathbb{R}^n}$ satisfy the weak localization \eqref{eq:wl_intro} and homogeneous approximation property \eqref{eq:hap_intro}; see \Cref{thm:framemeasure_rkhs}. In addition, they can also be shown to hold for general Gabor frames (cf. \Cref{sec:gabor}) and exponential frames on unbounded spectra (cf. \Cref{sec:exponential}), which do not necessarily satisfy \eqref{eq:hap_intro}. Hence, for all such frames, \Cref{thm:redundancy_intro} answers the question mentioned on \cite[p. 673]{balan2007measure} on the existence of subframes near the critical density.

\subsection{Methods}
\Cref{thm:main_intro}, \Cref{thm:exponential_intro}, and \Cref{thm:gabor_intro} follow directly from \Cref{thm:redundancy_intro} combined with the general discretization result for continuous frames obtained in \cite{freeman2019discretization}.
Our proof of \Cref{thm:redundancy_intro} is based on
 a recent selector variant of Weaver's conjecture obtained by the first author in \cite{bownik2024selector}; see \Cref{thm:selector}. This result shows that for a Bessel sequence with upper bound $1$ and any partition of its index set, there exists a selector set that yields a Bessel sequence with an upper bound that only depends on the norms of the vectors and the minimal cardinality of the sets in the partition. Applying this result to the canonical Parseval frame of an adequate subframe of the given frame, this allows us (by use of the identities in \eqref{eq:frd_intro}) to remove a set of positive Beurling density from the given frame yet leave a frame with a reduced, adequately controlled upper density (cf. \Cref{prop:main}).
By a finite iteration of removing sets of positive Beurling density, this yields eventually a subset with density at most $1+\varepsilon$ that yet leaves a frame.
See \Cref{sec:frame_critical} for the precise details.

\subsection{Organization} The paper is organized as follows. Section \ref{sec:frames} provides some preliminaries on abstract frames used throughout the main text. The preliminaries on reproducing kernels on metric measure spaces are discussed in Section \ref{sec:repro_measure}. Section \ref{sec:frame_measure_beurling} is devoted to proving the identities relating frame measure and Beurling density in the setting of reproducing kernel Hilbert spaces. Our main results are proven in Section \ref{sec:frame_critical}. Lastly, various examples to which our main results apply are provided in Section \ref{sec:examples}.

\section{Frames in Hilbert spaces} \label{sec:frames}
Let $(X, \mu)$ be a measure space and let $\mathcal{H}$ be a separable Hilbert space. A system of vectors $\{ g_x \}_{x \in X}$ is called a \emph{frame} for $\mathcal{H}$ if it satisfies the following conditions:
\begin{enumerate}
\item[(i)] For each $f \in \mathcal{H}$, the function $X \ni x \mapsto \langle f, g_x \rangle \in \mathbb{C}$ is measurable;
\item[(ii)] there exist constants $A, B > 0$, called \emph{frame bounds}, such that
\begin{align} \label{eq:frameineq}
A \| f \|^2 \leq \int_X |\langle f, g_x \rangle |^2 \; d\mu(x) \leq B \| f \|^2 \quad \text{for all} \quad f \in \mathcal{H}. 
\end{align} 
\end{enumerate}
A system $\{g_x \}_{x \in X}$ satisfying condition (i) and the upper bound in \eqref{eq:frameineq} is said to be \emph{Bessel system} in $\mathcal{H}$ with Bessel bound $B>0$. A frame $\{ g_x \}_{x \in X}$ for $\mathcal{H}$ whose frame bounds can be chosen to be $A=B=1$ is called a \emph{Parseval frame} for $\mathcal{H}$. 

We treat a frame $\{g_x \}_{x \in X}$ as a family indexed by $X$ and allow for repetitions.

If $\{ g_x \}_{x \in X}$ is a frame for $\mathcal{H}$, then its associated frame operator
\[
S = \int_X \langle \cdot , g_x \rangle g_x \; d\mu (x)
\]
is boundedly invertible on $\mathcal{H}$, and the system $\{S^{-1} g_x \}_{x \in X}$ is also a frame for $\mathcal{H}$, called the \emph{canonical dual frame} of $\{ g_x \}_{x \in X}$. Similarly, for a frame $\{ g_x \}_{x \in X}$, the system $\{S^{-1/2} g_x \}_{x \in X}$ is a Parseval frame for $\mathcal{H}$. 

The following result is  \cite[Corollary 7.7]{bownik2024selector}, which follows easily from \cite[Theorem 7.6]{bownik2024selector}.

\begin{theorem}[\cite{bownik2024selector}] \label{thm:discretization}
There exist a constants $c_0 > 0$ with the following property:

If $\{ g_x \}_{x \in X}$ is a frame for $\mathcal{H}$ with frame bounds $0 < A\leq B<\infty$ and there exists $C>0$ such that
\[
\| g_x \|^2 \leq C \quad \text{for a.e.} \quad x \in X,
\]
then for every $\delta > 0$ there exists a sequence $\{x_i \}_{i \in I}$ in $X$, where $I \subseteq \mathbb{N}$, such that $\{ g_{x_i} \}_{i \in I}$ is a frame for $\mathcal{H}$ with frame bounds $ (A-\delta) aC / \delta^2 $ and $(B+\delta) aC / \delta^2$, where $a>0$ is a constant such that $c_0 \leq a \leq 2 c_0$.
\end{theorem}

In addition to the discretization result (\Cref{thm:discretization}), we will also use the following selector variant of Weaver's conjecture. Theorem \ref{thm:selector} is a special case of \cite[Theorem 3.1]{bownik2024selector}, see also \cite[Theorem 3.3]{bownik2019syndetic}.

\begin{theorem}[\cite{bownik2024selector}] \label{thm:selector}
Let $I$ be countable and $\alpha, r>0$. Suppose that $\{ g_i \}_{i \in I}$ is a Bessel sequence in a separable Hilbert space $\mathcal{H}$ with Bessel bound $1$ satisfying
\[
 \| g_i \|^2 \leq \alpha \quad \text{for all} \quad i \in I. \]
Then, given a collection $\{ J_k \}_{k \in K}$ of disjoint subsets of $I$ with $\# J_k \geq r$ for all $k \in K$, there exists a \emph{selector} $J \subseteq \bigcup_{k \in K} J_k$ satisfying
\begin{align} \label{eq:selector1}
\# (J \cap J_k) = 1 \quad \text{for all} \quad k \in K
\end{align}
and such that $\{g_i\}_{i \in J}$ is a Bessel sequence with bound
$ \big(\frac{1}{\sqrt{r}} + \sqrt{\alpha} \big)^2$.
\end{theorem}

\section{Reproducing kernels on metric measure spaces} \label{sec:repro_measure}
This section considers the setting of reproducing kernels on metric measure spaces. 

\subsection{Assumptions}
We make the following assumptions on the metric measure spaces and the reproducing kernels.

\subsection*{Metric measure space} We denote by $(X, d, \mu)$ a metric measure space, which is a set $X$ equipped with a metric $d$ and a Borel measure $\mu$. We assume that $\mu(X) = \infty$, and that all balls $B_r (x) = \{ y \in X : d(y,x) < r \}$ satisfy $\mu(B_r(x)) < \infty$ for all $r > 0$ and $x \in X$. In addition, we make the following assumptions:
\begin{enumerate}[-]
 \item \emph{Nondegeneracy of balls}: There exists $r>0$ such that
 \begin{align}\label{eq:NDB}
  \inf_{x \in X} \mu(B_r(x)) > 0.
 \end{align}
\item \emph{Weak annular decay property}:
\begin{align} \label{eq:wad}
 \lim_{r\to \infty} \sup_{x \in X} \frac{\mu(B_{r+1} (x) \setminus B_r (x))}{\mu(B_r(x))} = 0.
\end{align}
\item \emph{Doubling at large scale}: There exists $r_0 > 0$ and  $C_d > 0$ such that for all $r\geq r_0$
\begin{align} \label{eq:doubling}
 \mu(B_{2r}(x)) \leq C_d \mu(B_r (x)) \quad \text{for all} \quad x \in X.
\end{align}
\end{enumerate}
The standing assumptions \eqref{eq:NDB} and \eqref{eq:wad} are the same as in \cite{fuehr2017density, mitkovsi2020density}. The weak annular decay property \eqref{eq:wad} does imply that the measure is \emph{locally} doubling at large scale (cf. \cite[Lemma 3.2]{fuehr2017density}), which is a similar but weaker condition than the doubling condition \eqref{eq:doubling} assumed throughout this paper.

There are various simple sufficient conditions on a metric measure space that imply the weak annular decay property. For example, if $d$ is a doubling measure and is a length function or satisfies certain geodesic properties, then $(X, d, \mu)$ satisfies the annular decay property. See the papers \cite{buckley1999maximal, auscher2013local, tessera2007volume} (and the references therein) for further details.

We mention the following two consequences of the weak annular decay property \eqref{eq:wad}.
\begin{lemma}[\cite{fuehr2017density}] \label{lem:wad}
 For all $r' > 0$, we have
 \[
  \lim_{r \to \infty} \sup_{x \in X} \frac{\mu(B_{r+r'}(x))}{\mu(B_r(x))} = 1
 \quad \text{and} \quad
  \lim_{r \to \infty} \sup_{x \in X} \frac{\mu(B_{r+r'} (x) \setminus B_{r-r'} (x))}{\mu(B_r(x))} = 0.
 \]
\end{lemma}

See \cite[Lemmas 3.2 and 3.5]{fuehr2017density} for the proofs of the statements in \Cref{lem:wad}.
Lastly, the following simple consequence of our standard assumptions will be used repeatedly.

\begin{lemma} \label{lem:uniform_radii}
\[
 \lim_{r \to \infty} \inf_{x \in X} \mu (B_r (x)) = \infty.
\]
\end{lemma}
\begin{proof}
By nondegeneracy of balls \eqref{eq:NDB}, there exists $c_0>0$ and $r>0$ such that $\mu(B_{r}(x)) \ge c_0$ for all $x\in X$.
On the contrary suppose that there exists a constant $c_1>0$ and a sequence $(x_n)_{n\in \N}$ such that $\mu(B_{nr}(x_n)) \le c_1$.
By Lemma \ref{lem:wad} there exists $N\in \N$ such that for all $k\ge N$ and all $x\in X$ we have
\begin{equation}\label{ur3}
\frac{\mu(B_{(k+2)r} (x) \setminus B_{(k-2)r} (x))}{\mu(B_{kr}(x))} < \frac{c_0}{c_1}.
\end{equation}
Let $k\ge N$ be the smallest integer such that $B_{(k+1)r}(x_N) \setminus B_{kr}(x_N) \ne \emptyset$. Such $k$ exists since $\mu(X)=\infty$. By the minimality of $k$, we have $\mu(B_{kr}(x_N))=\mu(B_{Nr}(x_N)) \le c_1$. Take any $y \in B_{(k+1)r}(x_N) \setminus B_{kr}(x_N)$. Since
\[
B_r(y) \subseteq B_{(k+2)r}(x_N) \setminus B_{(k-1)r}(x_N)
\]
and $\mu(B_r(y))\ge c_0$, we have
\[
\frac{\mu(B_{(k+2)r} (x_N) \setminus B_{(k-1)r} (x_N))}{\mu(B_{kr}(x_N))} \ge \frac{c_0}{c_1}.
\]
This yields a contradiction with \eqref{ur3}.
\end{proof}

\subsection*{Reproducing kernel}
We denote by $\mathcal{H} \subseteq L^2 (X, \mu)$ an infinite-dimensional \emph{reproducing kernel Hilbert space (RKHS)}, i.e., a closed subspace such that, for each $x \in X$, the point evaluation $\mathcal{H} \ni f \mapsto f(x) \in \mathbb{C}$ is a well-defined bounded linear function. Then for each $x \in X$, there exists $k_x \in \mathcal{H}$ such that $f(x) = \langle f, k_x \rangle $ for all $f \in \mathcal{H}$. The function $k_x$ is said to be the \emph{reproducing kernel at $x$}, and the \emph{reproducing kernel for $\mathcal{H}$} is defined as the function $k : X \times X \to \mathbb{C}$ given by
\[
 k(x,y) = \langle k_y, k_x \rangle.
\]
Note  $k_x(y) = \overline{k(x,y)} = k(y,x)$, and that each $f \in \mathcal{H}$ satisfies the reproducing formula
\[
 f(x) = \langle f, k_x \rangle = \int_X  f(y) k(x,y) \; d\mu(y), \quad x \in X.
\]
In particular, the system $\{k_x \}_{x \in X}$ is a Parseval frame for $\mathcal{H} \subseteq L^2 (X, \mu)$.

A condition on the reproducing kernel that we will always assume is the following:
\begin{enumerate}[-]
 \item \emph{Diagonal condition}: There exist constants $C_1, C_2 > 0$ such that
 \begin{align} \label{eq:dc}
  C_1 \leq k(x,x) \leq C_2
 \end{align}
for all $x \in X$.
\end{enumerate}

\subsection*{Localization conditions}
In addition to the conditions \eqref{eq:NDB}, \eqref{eq:wad}, and \eqref{eq:doubling} on the metric measure space and the reproducing kernel \eqref{eq:dc}, we will also consider localization properties of systems of vectors in the reproducing kernel Hilbert space. Specifically, given a system of vectors $\{g_x\}_{x \in X}$ in the reproducing kernel Hilbert space $\mathcal{H} \subseteq L^2 (X)$, we consider the conditions:
\begin{enumerate}[-]
\item \emph{Weak localization property}: For every $\varepsilon > 0$, there exists $r = r(\varepsilon) > 0$ such that
\begin{align} \label{eq:wl}
 \sup_{x \in X} \int_{X \setminus B_r (x)} |g_x(y)|^2 \; d\mu(y) < \varepsilon^2.
\end{align}
\item \emph{Homogeneous approximation property}: If $\Lambda \subseteq X$ is a countable set such that $\{ g_{\lambda} \}_{\lambda \in \Lambda}$ is a Bessel sequence in $\mathcal{H}$, then for every $\varepsilon > 0$, there exists $r = r(\varepsilon) > 0$ such that
\begin{align} \label{eq:hap}
 \sup_{x \in X} \sum_{\lambda \in \Lambda \setminus B_r (x)} |\gf (x)|^2 < \varepsilon^2.
\end{align}
\end{enumerate}
For the special case when $\{g_x\}_{x \in X}$ are the reproducing kernels in $\mathcal{H}$, the above two conditions are the same as those used in \cite{fuehr2017density}; see also \cite{mitkovsi2020density} for related conditions.

In contrast to the assumptions \eqref{eq:NDB}, \eqref{eq:wad}, \eqref{eq:doubling}, and \eqref{eq:dc}, we will not always assume the conditions \eqref{eq:wl} and \eqref{eq:hap} to hold; instead, it will be explicitly mentioned in each result whenever the conditions \eqref{eq:wl} or \eqref{eq:hap} are assumed.

\subsection{Sets and Beurling densities}
The \emph{lower Beurling density} and \emph{upper Beurling density} of a set $\Lambda \subseteq X$ are defined as
\begin{align} \label{eq:lower_density}
 D^-(\Lambda) = \liminf_{r \to \infty} \inf_{x \in X} \frac{\#(\Lambda \cap B_r (x))}{\mu(B_r(x))}
\end{align}
and
\begin{align} \label{eq:upper_density}
 D^+(\Lambda) = \limsup_{r \to \infty} \sup_{x \in X} \frac{\#(\Lambda \cap B_r (x))}{\mu(B_r(x))},
\end{align}
respectively. In addition, we will also make use of the following weighted versions of the Beurling densities, 
\begin{align} \label{eq:lower_density_dimension}
 \lbdf(\Lambda) := \liminf_{r \to \infty} \inf_{x \in X} \frac{\#(\Lambda \cap B_r (x))}{\int_{B_r(x)} k(y,y) \; d\mu (y)}
\end{align}
and
\begin{align} \label{eq:upper_density_dimension}
 \ubdf(\Lambda) := \limsup_{r \to \infty} \sup_{x \in X} \frac{\#(\Lambda \cap B_r (x))}{\int_{B_r(x)} k(y,y) \; d\mu(y)},
\end{align}
which are sometimes referred to as the \emph{dimension-free lower} and \emph{upper Beurling density}, respectively.
Note that the diagonal condition \eqref{eq:dc} implies the relation
\begin{align} \label{eq:density_relation}
 \frac{D^{\pm} (\Lambda)}{C_2}  \leq D^{\pm}_0 (\Lambda) \leq \frac{D^{\pm}(\Lambda)}{C_1} .
\end{align}
for any arbitrary countable set $\Lambda \subseteq X$. We also occasionally write $D^{\pm}_{\mu}$ to stress the dependence of the densities on the measure $\mu$.

A set $\Lambda \subseteq X$ is said to be \emph{relatively separated} if there exists $r_0 > 0$ such that for all $r\geq r_0$ there exists $ C_r > 0$ such that
\[
 \#(\Lambda \cap B_r (x)) \leq C_r \mu(B_r (x))
\]
for all $x \in X$. We say that a set $\Lambda \subseteq X$ is \emph{relatively dense} if there exists $r_0 > 0$ such that, for all $r\geq r_0$,
\[
 1 \leq \# (\Lambda \cap B_{r} (x))
\]
for all $x \in X$. 

The following lemma is \cite[Lemma 3.4]{fuehr2017density}.

\begin{lemma}[\cite{fuehr2017density}] \label{lem:cardinality_measure}
If $\Lambda \subseteq X$ is a relatively separated set, then for all sufficiently large $\rho > 0$, and $s>\rho, t>0$, there exists a constant $C = C(\rho, \Lambda) > 0$ such that
\begin{align} \label{eq:coverlemma}
 \# \big( \Lambda \cap (B_{s+t} (x) \setminus B_s (x)) \big) \leq C \mu(B_{s+t+\rho}(x) \setminus B_{s
 -\rho} (x))
\end{align}
for all $x \in X$.
\end{lemma}

We will also need several lemmas about Beurling densities.

\begin{lemma}\label{exist}
For any finite collection of sets $\Lambda \subseteq X$,
there exists a partition $\{X_y\}_{y\in Y}$ satisfying for some $R_0>0$ and $C_0>1$,
 \begin{align}\label{xy}
 B_{R_0}(y) & \subseteq X_y \subseteq B_{3R_0}(y) \qquad \text{for all }y\in Y,
\\
\label{eq:stephyp2}
 \frac{1}{C_0} D^-_{\mu}(\Lambda) & \leq \frac{\#(\Lambda \cap X_y)}{\mu(X_y)} \leq C_0 D^{+}_{\mu}(\Lambda) \qquad \text{for all } y \in Y.
\end{align}
\end{lemma}

\begin{proof}
Choose $R_0>0$ such that doubling property \eqref{eq:doubling} holds for $r\ge R_0$, and for every set $\Lambda$ in a finite collection we have
\[
\frac{1}{2} D^-_\mu(\Lambda) \le \frac{\#(\Lambda \cap B_{r}(x))}{\mu(B_{r}(x))} \le 2 D^+_{\mu}(\Lambda) \qquad\text{for all }x\in X. 
\]
Let  $\{B_{R_0}(y)\}_{y\in Y}$ be a maximal disjoint family of balls.  Hence, $\bigcup_{y\in Y} B_{2R_0}(y)=X$. Choose a partition $(X_y)_{y\in Y}$ of $X$ such that
\[
 B_{R_0}(y)  \subseteq X_y \subseteq B_{2R_0}(y) \qquad \text{for all }y\in Y.
\]
Then, for any $y\in Y_0:=Y$ we have 
\[
\begin{aligned}
 \frac{1}{2C_d} D^-_{\mu}(\Lambda)  \leq \frac{1}{C_d} \frac{\#(\Lambda \cap B_{R_0}(y))}{\mu(B_{R_0}(y))} 
& \le  \frac{\#(\Lambda \cap X_y)}{\mu(X_y)} 
\le C_d \frac{\#(\Lambda \cap B_{2R_0}(y))}{\mu(B_{2R_0}(y))}
\leq 2 C_d D^{+}_{\mu}(\Lambda).
\end{aligned}
\]
Thus, the partition $\{X_y\}_{y\in Y}$ satisfies \eqref{xy} and \eqref{eq:stephyp2} with $C_0=2C_d$.
\end{proof}

The following regularization result, which might be of independent interest, improves upon Lemma \ref{exist}. 

\begin{lemma}\label{regularization}
 Let $\Lambda \subseteq X$ be such that $0<D^-_{\mu}(\Lambda) \le D^+_{\mu}(\Lambda)<\infty$. Let $\ve>0$. Suppose that $\{X_y\}_{y\in Y}$ is a partition of $X$ satisfying satisfying \eqref{xy} and \eqref{eq:stephyp2} for some $R_0>0$ and $C_0>1$.
  Then, for sufficiently large $R'>0$, there exists a partition $\{X'_{y}\}_{y \in Y'}$ of $X$ satisfying:
 \begin{align}\label{xyh}
 B_{R'}(y) & \subseteq X'_{y} \subseteq B_{3R'}(y) \qquad \text{for all }y\in Y',
 \\
 \label{reg0}
 D^-_{\mu}(\Lambda) - \ve &\le \frac{\#(\Lambda \cap X'_{y})}{\mu(X'_{y})} \le D^{+}_{\mu}(\Lambda) + \ve \qquad \text{for all }y\in Y',
 \\
 \label{nest}
&  \forall y \in Y\quad  \exists y' \in Y' \quad X_y \subseteq X'_{y'}.
\end{align} 
 \end{lemma}

Since the proof of Lemma \ref{regularization} is quite long and complex, it is postponed till the Appendix. Using Lemma \ref{regularization} we deduce two convenient results on Beurling densities.

\begin{lemma}\label{sub}
Let $\Lambda \subseteq X$ be such that $0<D^-_{\mu}(\Lambda) \le D^+_{\mu}(\Lambda)<\infty$. Then for any $\ve>0$ there exists $\Gamma \subseteq \Lambda$ such that
\[
D^-_{\mu}(\Lambda) - \ve \le D^-_{\mu}(\Gamma) \le D^+_{\mu}(\Gamma) \le D^-_{\mu}(\Lambda).
\]
 \end{lemma}
 
 \begin{proof}
By Lemma \ref{exist}, there exists a partition $\{X_y\}_{y\in Y}$ of $X$ satisfying \eqref{xy} and \eqref{eq:stephyp2}.
 Let $\{X'_{y}\}_{y' \in Y'}$ be a partition of $X$ as in Lemma \ref{regularization} for some $R'>0$. Since $R'>0$ can be chosen to be arbitrarily large, by Lemma \ref{lem:uniform_radii}, we can assume that
 \[
 \mu(X'_y) \ge \mu(B_{R'}(y)) \ge 1/\ve \qquad\text{for all } y\in Y'.
 \]
 If $y \in Y'$ is such that $D^- (\Lambda) \mu(X'_y) \geq \# (\Lambda \cap X_y)$, then we set $\Lambda_y = \Lambda \cap X'_y$. Otherwise, using \eqref{reg0},  we choose a set $\Lambda_y \subseteq \Lambda \cap X'_y$ of cardinality
 $\#\Lambda_y = \lfloor D^-_{\mu}(\Lambda) \mu(X_y') \rfloor$. In either case,
 \begin{align*} 
  D^- (\Lambda) - \varepsilon  \leq \frac{\# \Lambda_y}{\mu(X'_y)} \leq D^-_{\mu} (\Lambda).
 \end{align*}
Define $\Gamma := \bigcup_{y \in Y} \Lambda_y$. Then
\begin{equation}\label{eq:Lambdaycardinality}
  D^- (\Lambda) - \varepsilon \leq \frac{\# (\Gamma \cap X'_y)}{\mu(X'_y)} \leq D^-_{\mu} (\Lambda) \qquad \text{for all } y \in Y'.
 \end{equation}
By \eqref{xyh}, we get for $x \in X$ and $r > 6R'$,
\begin{align} \label{eq:nested1}
 B_r(x) \subseteq \bigcup_{y \in Y' :\  X'_y \cap B_{r} (x) \ne \emptyset} X'_y \subseteq B_{r+6R'} (x)
\end{align}
and
\begin{align} \label{eq:nested2}
 B_{r-6R'} (x) \subseteq \bigcup_{y \in Y' :\  X'_y \subseteq B_{r} (x)} X'_y \subseteq B_{r} (x).
\end{align}
Using \eqref{eq:nested1}, we thus get
\begin{align*}
 \frac{\#(\Gamma \cap B_r (x))}{\mu(B_r(x))} &\leq \frac{\#\big(\Gamma \cap \bigcup_{y \in Y' : X'_y \cap B_{r} (x) \ne \emptyset} X'_y \big)}{\mu \big(\bigcup_{y \in Y' : X'_y \cap B_{r} (x) \ne \emptyset} X'_y\big)}  \frac{\mu( B_{r+6R'}(x))}{\mu(B_r(x))}.
\end{align*}
Since $\lim_{r \to \infty} \sup_{x \in X} \frac{\mu( B_{r+6R'}(x))}{\mu(B_r(x))} = 1$, by the weak annular decay property \eqref{eq:wad}, the claim $D_\mu^+(\Gamma) \leq D^-_{\mu}(\Lambda)$ follows from \eqref{eq:Lambdaycardinality}.
The claim $D^-_{\mu} (\Lambda) - \varepsilon \leq D^-_{\mu} (\Gamma)$ follows in a similar manner from \eqref{eq:nested2}.
  \end{proof}

\begin{lemma}\label{sup}
  For $i=1,2$, let $\Lambda_i \subseteq X$ be such that
 \[
 0<D^-_{\mu}(\Lambda_i) \le D^+_{\mu}(\Lambda_i)<\infty
 \qquad\text{and}\qquad D^+_{\mu}(\Lambda_1) - D^-_{\mu}(\Lambda_1) \le D^-_{\mu}(\Lambda_2).
\] 
 Then for any $\ve>0$, there exists a set $\Gamma$ satisfying $\Lambda_1 \subseteq \Gamma \subseteq \Lambda_1 \cup \Lambda_2$ and such that
\[
D^+_{\mu}(\Lambda_1) - \ve \le D^-_{\mu}(\Gamma) \le D^+_{\mu}(\Gamma) \le D^+_{\mu}(\Lambda_1)+\ve.
\]
 \end{lemma}
\begin{proof}
By Lemma \ref{exist}, there exist $R_0>0$ and $C_0>1$, and a partition $\{X_y\}_{y\in Y}$ such that \eqref{xy} and \eqref{eq:stephyp2} hold both for $\Lambda_1$ and $\Lambda_2$.
By Lemma \ref{regularization} applied to the set $\Lambda_1$, there exists a partition $\{X'_{y}\}_{y \in Y'}$  of $X$ satisfying \eqref{xyh}
for some radius $R'>0$ and
 \begin{equation}\label{ep2}
  D^-_{\mu} (\Lambda_1) - \frac{\varepsilon}{2} \leq \frac{\#(\Lambda_1 \cap X'_y)}{\mu(X'_y)} \leq D^+_{\mu} (\Lambda_1) + \frac{\varepsilon}{2} \qquad \text{for all } y \in Y'.
 \end{equation}
 Since \eqref{eq:stephyp2} holds for the set $\Lambda_2$ and the partition $\{X_y\}_{y\in Y}$, it also holds for a coarser partition $\{X'_y\}_{y\in Y'}$ by \eqref{nest} as any set $X'_{y}$ can be written as a disjoint union of some $X_z$ with $z \in Y$.
 Applying Lemma \ref{regularization} to the set $\Lambda_2$ and the partition $\{X'_{y}\}_{y \in Y'}$ yields a partition $\{\bar X_y\}_{y\in \bar Y}$ of $X$ satisfying for some $\bar R>0$,
  \begin{align}\label{xyz}
 B_{\bar R}(y) & \subseteq \bar X_{y} \subseteq B_{3\bar R}(y) \qquad \text{for all }y\in \bar Y,
 \\
 \label{ep3}
 D^-_{\mu}(\Lambda_2) - \frac\ve{2} &\le \frac{\#(\Lambda_2 \cap \bar X_{y})}{\mu(\bar X_{y})} \le D^{+}_{\mu}(\Lambda_2) + \frac\ve{2} \qquad \text{for all }y\in \bar Y,
 \\
 \label{nestz}
&  \forall y \in Y'\quad  \exists \bar y \in \bar Y \quad X'_y \subseteq \bar X_{\bar y}.
\end{align} 
Since $\bar R>0$ can be chosen to be arbitrarily large, by Lemma \ref{lem:uniform_radii}, we can assume that
 \begin{equation}\label{ep4}
 \mu(\bar X_y) \ge \mu(B_{\bar R}(y)) \ge \frac 1{\ve} \qquad\text{for all } y\in \bar Y.
 \end{equation}
By the nestedness property \eqref{nestz}, we can write each $\bar{X}_y$ as a disjoint union of sets $X'_{y'}$ for certain $y' \in Y'$. Hence, it follows from \eqref{ep2} that
 \begin{equation}\label{ep5}
  D^-_{\mu} (\Lambda_1) - \frac{\varepsilon}{2} \leq \frac{\#(\Lambda_1 \cap \bar X_y)}{\mu(\bar X_y)} \leq D^+_{\mu} (\Lambda_1) + \frac{\varepsilon}{2} \qquad \text{for all } y \in \bar Y.
 \end{equation}
 Combining \eqref{ep3} and \eqref{ep5} with the assumption that $D^-_{\mu}(\Lambda_1) + D^-_{\mu}(\Lambda_2) \ge D^+_{\mu}(\Lambda_1)$ yields
\[
 D^+_{\mu} (\Lambda_1) - \varepsilon \leq \frac{\#(\Lambda_1 \cap \bar X_y)+ \#(\Lambda_2 \cap \bar X_y)}{\mu(\bar X_y)}  \qquad \text{for all } y \in \bar Y.
\]

By  the upper bound in \eqref{ep5} for any $y \in \bar Y$, we can choose a set $\Lambda_y \subseteq \Lambda_2 \cap \bar X_y$ such that 
 \begin{align*} 
  D_\mu^+ (\Lambda_1) - \varepsilon \leq \frac{ \#(\Lambda_1 \cap \bar X_y) + \# \Lambda_y  }{\mu(\bar X_y)} \leq D^+_{\mu} (\Lambda_1) + \frac{\varepsilon}2.
 \end{align*}
 Indeed, taking $\Lambda_y=\emptyset$ guarantees the upper bound, whereas taking $\Lambda_y=  \Lambda_2 \cap \bar X_y$ yields the lower bound. By \eqref{ep4} there exists an intermediate choice of $\Lambda_y$, where both bounds hold.
Define $\Gamma := \Lambda_1 \cup \bigcup_{y \in \overline{Y}} \Lambda_y$.
Then similar arguments as in the proof of  Lemma \ref{sub} yields the required conclusion.
\end{proof}

\section{Frame measure and Beurling density} \label{sec:frame_measure_beurling}
Throughout this section, we denote by $\mathcal{H} \subseteq L^2 (X, \mu)$ a reproducing kernel Hilbert space satisfying the assumptions outlined in \Cref{sec:repro_measure}.

Let $\Lambda \subseteq X$ be countable.
If $\{\gf\}_{\lambda \in \Lambda}$ is a frame for $\mathcal{H}$ with canonical dual frame $\{\df\}_{\lambda \in \Lambda}$, then we define the associated \emph{lower frame measure} and \emph{upper frame measure} by
\begin{align*}
 M^-(\{\gf \}_{\lambda \in \Lambda})  := \liminf_{r \to \infty} \inf_{x \in X} \frac{1}{\#(\Lambda \cap B_r (x))} \sum_{\lambda \in \Lambda \cap B_r (x)} \langle \gf, \df \rangle
\end{align*}
and
\begin{align*}
 M^+(\{\gf \}_{\lambda \in \Lambda})  := \limsup_{r \to \infty} \sup_{x \in X} \frac{1}{\#(\Lambda \cap B_r (x))} \sum_{\lambda \in \Lambda \cap B_r (x)} \langle \gf, \df \rangle,
\end{align*}
 respectively. This notion of frame measure forms a special case of the general notion considered in \cite{balan2007measure}; see also \cite{balan2006density, caspers2023overcompleteness} for closely related notions of frame measures.

We make the following definition.

\begin{definition}
Let $\{g_x \}_{x \in X}$ be a system of vectors in $\mathcal{H}$. We say that $\{g_x \}_{x \in X}$ satisfies the \emph{frame measure/density property} if for each countable set $\Lambda \subseteq X$ such that $\{\gf \}_{\lambda \in \Lambda}$ is a frame for $\mathcal{H}$ the following formulae hold:
  \begin{align} \label{eq:frd}
  M^-(\{\gf \}_{\lambda \in \Lambda}) = \frac{1}{D_0^+ (\Lambda)} \quad \text{and} \quad M^+(\{\gf \}_{\lambda \in \Lambda}) = \frac{1}{D_0^- (\Lambda)}.
 \end{align}
 A frame $\{\gf \}_{\lambda \in \Lambda}$ satisfying \eqref{eq:frd} is said to satisfy the \emph{frame measure/density formulae}.
\end{definition}

The aim of this section is to show that frames satisfying the weak localization property \eqref{eq:wl} and the homogeneous approximation property \eqref{eq:hap} satisfy the frame measure/density property \eqref{eq:frd}. For this, we start with the following lemma.

\begin{lemma} \label{lem:relatively_dense}
Let $\{ g_x \}_{x \in X}$ be a system of vectors in $\mathcal{H}$ such that $\inf_{x \in X} \| g_x \|^2 >0$ and that satisfies the weak localization property \eqref{eq:wl}.

 If $\Lambda \subseteq X$ is countable such that $\{\gf \}_{\lambda \in \Lambda}$ is a Bessel sequence, then $\Lambda$ is relatively separated. If, moreover, $\{ g_x \}_{x \in X}$ satisfies the homogeneous approximation property \eqref{eq:hap} and $\{\gf \}_{\lambda \in \Lambda}$ is a frame for $\mathcal{H}$, then $\Lambda$ is also relatively dense.
\end{lemma}
\begin{proof} 
The fact that $\Lambda$ is relatively separated whenever $\{\gf \}_{\lambda \in \Lambda}$ is a Bessel sequence follows from an adaptation of the proof of \cite[Lemma 3.7]{fuehr2017density}, which concerns only Bessel sequences of reproducing kernels.

Indeed, let $\varepsilon^2= \tfrac12 \inf_{x\in X} ||g_x||^2$. By the diagonal condition \eqref{eq:dc}, there exist constants $C_1,C_2 > 0$ such that $C_1 \le ||k_x||^2 \le C_2$ for all $x \in X$. 
By the weak weak localization property \eqref{eq:wl}, there exists $r_1=r_1(\ve)>0$ such that for all $x\in X$,
\[
\int_{X \setminus B_{r_1}(x)} |g_x(y)|^2 \; d\mu(y) \le \ve^2.
\]
Thus,
\begin{equation}\label{f1}
\int_{B_{r_1}(x)} |g_x(y)|^2 \; d\mu(y) = ||g_x||^2 - \int_{X \setminus B_{r_1}(x)} |g_x(y)|^2 \; d\mu(y) \ge \ve^2.
\end{equation}
Since $\{\gf \}_{\lambda \in \Lambda}$ is a Bessel sequence, there exists $B>0$ such that for all $y\in X$,
\begin{equation}\label{f2}
\sum_{\lambda \in \Lambda} |g_\lambda(y)|^2 = \sum_{\lambda \in \Lambda} |\lan g_\lambda, k_y \ran|^2 \le B ||k_y||^2 \le BC_2.
\end{equation}
Since the measure $\mu$ is doubling at large scales \eqref{eq:doubling}, there exists $r_2>0$ and $C_d > 0$ such that for all $r\ge r_2$,
\[
\mu(B_{2r}(x)) \le C_d \mu(B_r(x)) \qquad\text{for all }x\in X.
\]
Let $x\in X$ and $r\ge \max\{r_1,r_2\}$. By \eqref{f1} and \eqref{f2}, we have
\[
\begin{aligned} 
\ve^2 \#(\Lambda \cap B_r(x)) & \le \sum_{\lambda \in \Lambda \cap B_r(x)} \int_{B_r(\lambda)} |g_\lambda(y)|^2 \; d\mu(y) 
\le \sum_{\lambda \in \Lambda \cap B_r(x) }\int_{B_{2r}(x)} |g_\lambda(y)|^2 \; d\mu(y) 
\\
& \le \int_{B_{2r}(x)}  \sum_{\lambda \in \Lambda }|g_\lambda(y)|^2 \; d\mu(y) \le B C_2 C_d \mu(B_{r}(x)).
\end{aligned}
\]
This shows that $\Lambda$ is relatively separated.

Suppose next that $\{\gf \}_{\lambda \in \Lambda}$ is a frame with lower frame bound $A > 0$. Let $0 < \varepsilon' < A C_1$ be arbitrary.  Then, by the homogeneous approximation property \eqref{eq:hap}, there exists $r_0 = r_0(\varepsilon') > 0$ such that
\[
 \sup_{x \in X} \sum_{\lambda \in \Lambda \setminus B_{r_0} (x)} |\langle \gf, k_{x} \rangle |^2 =  \sup_{x \in X} \sum_{\lambda \in \Lambda \setminus B_{r_0} (x)} | \gf(x)|^2 < \varepsilon'.
\]
Arguing by contradiction, assume that $\Lambda$ is not relatively dense. Then, given $r \geq r_0$, there exists $x \in X$ such that $\Lambda \cap B_r (x) = \emptyset$. Hence,
\[
 A C_1 \leq A \| k_x \|^2 \leq \sum_{\lambda \in \Lambda} |\langle \gf, k_{x} \rangle |^2 = \sum_{\lambda \in \Lambda \setminus B_r (x)} |\langle \gf, k_{x} \rangle |^2 < \varepsilon',
\]
which is a contradiction.
\end{proof}

The following theorem is the main result of this section.

\begin{theorem} \label{thm:framemeasure_rkhs}
Let $\{g_x\}_{x \in X}$ be a system of vectors in $\mathcal{H}$ satisfying $\inf_{x \in X} \| g _x \|^2 > 0$, the weak localization  \eqref{eq:wl}, and the homogeneous approximation property \eqref{eq:hap}.
Suppose $\Lambda \subseteq X$ is a countable set such that $\{\gf \}_{\lambda \in \Lambda}$ is a frame for $\mathcal{H}$. Then $\{\gf \}_{\lambda \in \Lambda}$ satisfies 
  \[ M^-(\{\gf \}_{\lambda \in \Lambda}) = \frac{1}{D_0^+ (\Lambda)} \quad \text{and} \quad M^+(\{\gf \}_{\lambda \in \Lambda}) = \frac{1}{D_0^- (\Lambda)}, \]
so that  $\{g_x\}_{x \in X}$ satisfies the frame measure/density property \eqref{eq:frd}.
\end{theorem}
\begin{proof}
Our proof is modelled on the proofs of \cite[Theorem 2.2]{fuehr2017density}, \cite[Theorem 4.1]{mitkovsi2020density} and \cite[Theorem 3.2]{caspers2023overcompleteness}; see also \cite{kolountzakis1996structure, balan2006density, iosevich2006weyl, olevskii2012revisiting} for similar proof ideas.

Let $A, B > 0$ be the frame bounds of $\{\gf\}_{\lambda \in \Lambda}$. Let $\{h_\lambda\}_{\lambda \in \Lambda}$ be its canonical dual frame.
 Fix $\varepsilon > 0$ and $x \in X$, and choose $r' = r'(\varepsilon) > 0$ such that the weak localization \eqref{eq:wl} and the homogeneous approximation property \eqref{eq:hap} assumptions are satisfied. Let $\rho > 0$ be sufficiently large so that the conclusion of \Cref{lem:cardinality_measure} holds. Lastly, since $\{\gf \}_{\lambda \in \Lambda}$ is a frame, we can choose by \Cref{lem:relatively_dense} some $r_0 > 0$ such that for all $r > r_0$
 \[
  1 \leq \# (\Lambda \cap B_{r}(x)) \leq C \mu(B_r(x)) , \quad x \in X,
 \]
where $C = C(r) > 0$ only depends on $r$. Throughout the proof, we let $r > 0$ be sufficiently large so that $r> r_0$ and $r - r' > \rho$.

For $y \in X$, define $F (y) = \sum_{\lambda \in \Lambda} \gf(y) \overline{h_{\lambda}(y)}$, and
\begin{align*}
F_1(y) = \sum_{\lambda \in \Lambda \cap B_{r-r'}(x)} \gf(y) \overline{h_{\lambda}(y)}, \quad F_2 (y) = \sum_{\lambda \in \Lambda \setminus B_{r+r'}(x)}\gf(y) \overline{h_{\lambda}(y)}
\end{align*}
and
\begin{align*}
 F_3 (y) = \sum_{\lambda \in \Lambda \cap B_{r+r'} (x) \setminus B_{r-r'}(x)} \gf(y) \overline{h_{\lambda}(y)},
\end{align*}
so that $F = \sum_{i = 1}^3 F_i$. We consider four steps.
\\~\\
\textbf{Step 1.} In this step, we provide estimates for $\int_{B_r(x)} F_i (y) \; d\mu(y)$ for $i = 1, 2, 3$ by following estimates in the proof of \cite[Theorem 2.2]{fuehr2017density}. 

First, note that
\begin{align*}
 \int_{B_r(x)} F_1 (y) \; d\mu(y) = \int_X \sum_{\lambda \in \Lambda \cap B_{r-r'} (x)}  \gf(y) \overline{h_{\lambda} (y)} \; d\mu(y) - L = \sum_{\lambda \in \Lambda \cap B_{r-r'} (x)}  \langle \gf ,h_{\lambda} \rangle - L ,
 \end{align*}
 where
 \begin{align*}
  L := \sum_{\lambda \in \Lambda \cap B_{r-r'} (x)} \int_{X \setminus B_r(x)} g_{\lambda} (y) \overline{h_{\lambda} (y)} \; d\mu(y).
\end{align*}
For $\lambda \in \Lambda \cap B_{r-r'}(x)$ and $y \in X \setminus B_r(x)$, we have that $d(\lambda, y )>r'$. As such, it follows by the weak localization assumption \eqref{eq:wl} that
\[
\bigg| \int_{X \setminus B_r(x)} \gf(y) \overline{\df(y)} \; d\mu(y) \bigg| \leq \bigg(\int_{X \setminus B_{r'}(\lambda)} |\gf(y) |^2 \; d\mu(y) \bigg)^{1/2} \| \df \|_{L^2} \leq \varepsilon A^{-1/2}
\]
and thus
$
 |L| \leq \varepsilon A^{-1/2} \# (\Lambda \cap B_r (x)).
$

For estimating $\int_{B_r(x)} F_2 \; d\mu$, let $y \in B_r(x)$ and $\lambda \in \Lambda \setminus B_{r+r'}(x)$, so that $d(\lambda, y ) > r'$. Then it follows from the homogeneous approximation property \eqref{eq:hap} that
\[
\sum_{\lambda \in \Lambda \setminus B_{r+r'}(x)} |\gf(y)|^2 \leq \sum_{\lambda \in \Lambda \setminus B_{r'}(y)} |\gf(y)|^2 < \varepsilon^2.
\]
On the other hand, note that
\[
\sum_{\lambda \in \Lambda} |\df(y)|^2 = \sum_{\lambda \in \Lambda} |\langle k_y, \df \rangle|^2 \leq A^{-1} \| k_y \|^2 \leq A^{-1} C_2,
\]
where $C_2 > 0$ is the constant in the diagonal condition \eqref{eq:dc}. In combination, this gives
\begin{align*}
 \bigg| \int_{B_r(x)} F_2 (y) \; d\mu(y) \bigg| &\leq \int_{B_r(x)} \bigg(  \sum_{\lambda \in \Lambda \setminus B_{r+r'}(x)} |\gf(y)|^2 \bigg)^{1/2} \bigg( \sum_{\lambda \in \Lambda} |\df(y) |^2 \bigg)^{1/2} \; d\mu(y) \\
 &\leq \varepsilon (A^{-1} C_2)^{1/2} \mu(B_r(x)).
\end{align*}

Lastly, we estimate $\int_{B_r(x)} F_3 \; d\mu$ by
\begin{align*}
 \int_{B_r(x)} |F_3(y)| \; d\mu(y) &\leq \sum_{\lambda \in \Lambda \cap B_{r+r'}(x) \setminus B_{r-r'}(x)} \int_X |\gf(y) ||\df(y)| \; d\mu(y) \\
 &\leq \sum_{\lambda \in \Lambda \cap B_{r+r'}(x) \setminus B_{r-r'}(x)} \| \gf \| \| \df \| \\
 &\leq B^{1/2} A^{-1/2}  \# (\Lambda \cap (B_{r+r'} (x) \setminus B_{r-r} (x))).
\end{align*}
\\
\textbf{Step 2.} Using the same notation as in Step 1, we can write
\[
 \sum_{\lambda \in \Lambda \cap B_{r-r'}(x)} \langle \gf, \df \rangle = \int_X F_1(y) \; d\mu(y) = \int_{B_r(x)} (F(y) - F_2(y) - F_3(y)) \; d\mu(y) + L.
\]
Therefore, by using the estimates in Step 1, it follows that
\begin{align*}
 &\bigg| \int_{B_r(x)} F(y) \; d\mu(y) - \sum_{\lambda \in \Lambda \cap B_r (x)} \langle \gf , \df \rangle \bigg| \\
 &\quad \quad = \bigg| \int_{B_r(x)} F(y) \; d\mu(y) - \sum_{\lambda \in \Lambda \cap B_{r-r'} (x)} \langle \gf , \df \rangle - \sum_{\lambda \in \Lambda \cap (B_{r}(x) \setminus B_{r-r'} (x))} \langle \gf , \df \rangle \bigg| \\
 &\quad \quad \leq \bigg| \int_{B_r(x)} F_2 (y) \; d\mu(y) \bigg| + \bigg| \int_{B_r(x)} F_3 (y) \; d\mu(y) \bigg| + |L| + \sum_{\lambda \in \Lambda \cap (B_{r}(x) \setminus B_{r-r'} (x))} \langle \gf , \df \rangle
 \\
 &\quad \quad \leq  \varepsilon (A^{-1} C_2)^{1/2} \mu(B_r(x)) + B^{1/2} A^{-1/2} \# (\Lambda \cap (B_{r+r'} (x) \setminus B_{r-r'} (x))) \\
  &\quad \quad \quad \quad + \varepsilon A^{-1/2} \# (\Lambda \cap B_r (x)) + \# (\Lambda \cap (B_{r+r'} (x) \setminus B_{r-r'} (x))) ,
\end{align*}
where the last step used that $\langle \gf, h_{\lambda} \rangle \leq 1$ for all $\lambda \in \Lambda$.
 By \Cref{lem:cardinality_measure}, there exists $C' = C'(\rho, \Lambda) > 1$ such that
\[
 \# (\Lambda \cap B_{r+r'} (x) \setminus B_{r-r'} (x)) \leq C' \mu(B_{r+r'+\rho}(x) \setminus B_{r-r'-\rho}(x)).
\]
Hence,
\begin{align*}
 &\bigg| \int_{B_r(x)} F(y) \; d\mu(y) - \sum_{\lambda \in \Lambda \cap B_r (x)} \langle \gf, \df \rangle \bigg| \\
 &\quad \quad \leq  \varepsilon (A^{-1}C_2)^{1/2} \mu(B_r(x)) + (1+ B^{1/2}A^{-1/2}) C' \mu ( B_{r+r'+\rho} (x) \setminus B_{r-r'-\rho} (x)) \\
 &\quad \quad \quad + \varepsilon A^{-1/2} \# (\Lambda \cap B_r (x)), \numberthis \label{eq:keystep2}
\end{align*}
which is the key estimate to be used in the next steps.
\\~\\
\textbf{Step 3.}
Using that $\{\gf\}_{\lambda \in \Lambda}$ and $\{\df\}_{\lambda \in \Lambda}$ are canonical dual frames, it follows that
\[
 k(y,y) = \langle k_y, k_y \rangle = \sum_{\lambda \in \Lambda} \langle \gf , k_y \rangle \langle k_y, \df \rangle
 = \sum_{\lambda \in \Lambda} \gf (y) \overline{\df (y)} = F(y)
\]
for $y \in X$.

First, note that the estimate \eqref{eq:keystep2}, combined with the fact that $\langle \gf, \df \rangle \leq 1$ for all $\lambda \in \Lambda$, gives
\begin{align*}
 &\frac{1}{\mu(B_r(x))} \int_{B_r(x)} k(y,y) \; d\mu(y) \\
 &\quad \quad \leq \varepsilon (A^{-1}C_2)^{1/2}  + (1+ B^{1/2}A^{-1/2}) C' \frac{\mu  (B_{r+r'+\rho} (x) \setminus B_{r-r'-\rho} (x))}{\mu(B_r(x))} \\
 &\quad \quad \quad \quad  + \big(1 +\varepsilon A^{-1/2} \big) \frac{\# (\Lambda \cap B_r (x))}{\mu(B_r(x))}
\end{align*}
By Lemma \ref{lem:wad}, we have that
\begin{align} \label{eq:annularinproof}
 \lim_{r\to\infty} \sup_{x \in X} \frac{\mu (B_{r+r'+\rho} (x) \setminus B_{r-r'-\rho} (x))}{\mu(B_r(x))}  = 0.
\end{align}
Thus, the above observations easily yield
that
\begin{align} \label{eq:lower_density_proof}
 D^-(\Lambda) \geq \liminf_{r \to \infty} \inf_{x \in X} \frac{1}{\mu(B_r(x))} \int_{B_r(x)} k(y,y) \; d\mu(y) \geq C_1 > 0,
\end{align}
where $C_1 > 0$ is as in the diagonal condition \eqref{eq:dc}.

In addition to the above, the estimate \eqref{eq:keystep2} also yields that
\begin{align*}
& \bigg| \frac{\int_{B_r(x)} k(y,y) \; d\mu(y)}{\#(\Lambda \cap B_r(x))}- \frac{1}{\#(\Lambda \cap B_r(x))} \sum_{\lambda \in \Lambda \cap B_r (x)} \langle \gf, \df \rangle \bigg| \\
 &\quad \quad \leq  \varepsilon A^{-1/2} +  \varepsilon (A^{-1}C_2)^{1/2} \bigg(\frac{\# (\Lambda \cap B_r (x))}{\mu(B_r(x))}\bigg)^{-1} \\
 &\quad \quad \quad \quad + (1+ B^{1/2} A^{-1/2}) C' \frac{\mu ( (B_{r+r'+\rho} (x) \setminus B_{r-r'-\rho} (x)))}{\mu(B_r(x))}  \bigg(\frac{\# (\Lambda \cap B_r (x))}{\mu(B_r(x))}\bigg)^{-1}.
\end{align*}
Therefore, using again \eqref{eq:annularinproof} and the fact that
\[
 \limsup_{r \to \infty} \sup_{x \in X} \bigg(\frac{\# (\Lambda \cap B_r (x))}{\mu(B_r(x))}\bigg)^{-1}  = \frac{1}{D^-(\Lambda)}
\]
by \eqref{eq:lower_density_proof}, it follows that
\begin{align*}
 &\limsup_{r \to \infty} \sup_{x \in X} \bigg| \frac{\int_{B_r(x)} k(y,y) \; d\mu(y)}{\#(\Lambda \cap B_r(x))} - \frac{1}{\#(\Lambda \cap B_r(x))} \sum_{\lambda \in \Lambda \cap B_r (x)} \langle \gf, \df \rangle \bigg| = 0,
\end{align*}
as desired.
\\~\\
\textbf{Step 4.}
Lastly, we show that $M^+(\{\gf \}_{\lambda \in \Lambda}) = 1 / D_0^-(\Lambda)$; the argument for $D_0^+ (\Lambda) = 1/M^- (\{\gf \}_{\lambda \in \Lambda})$ is similar.
For $n \in \mathbb{N}$, choose  $r_n \in (0, \infty)$ increasing and points $x_n \in X$ such that
\[
 M^+ (\{\gf \}_{\lambda \in \Lambda}) = \lim_{n \to \infty} \frac{1}{\#(\Lambda \cap B_{r_n} (x_n))} \sum_{\lambda \in \Lambda \cap B_{r_n} (x_n)} \langle g_{\lambda}, h_{\lambda} \rangle .
\]
Using the conclusion of Step 3, this yields that
\begin{align*}
 M^+ (\{\gf \}_{\lambda \in \Lambda}) = \lim_{n \to \infty} \frac{\int_{B_{r_n}(x_n)} k(y,y) \; d\mu(y)}{\#(\Lambda \cap B_{r_n}(x_n))} \leq \frac{1}{D_0^-(\Lambda)}.
\end{align*}
Conversely, let $(r_n)_{n \in \mathbb{N}}$ be an increasing sequence in $(0, \infty)$ and let $(x_n)_{n \in \mathbb{N}}$ be points in  $X$ such that
\[
D^-_0(\Lambda) = \lim_{n \to \infty} \frac{\# (\Lambda \cap B_{r_n} (x_n))}{\int_{B_{r_n}(x_n)} k(y,y) \; d\mu(y)}. 
\]
Then
\[
\frac{1}{D_0^-(\Lambda)} = \lim_{n \to \infty} \frac{1}{\#(\Lambda \cap B_{r_n} (x_n))} \sum_{\lambda \in \Lambda \cap B_{r_n} (x_n)} \langle g_{\lambda}, h_{\lambda} \rangle \leq M^+ (\{\gf\}_{\lambda \in \Lambda}),
\]
which finishes the proof.
\end{proof}

We finish this section by stating two simple consequences of the frame measure/density formulae.

\begin{lemma} \label{lem:density_rkhs}
 If  $\{\gf \}_{\lambda \in \Lambda}$ is a frame satisfying the frame measure/density formulae \eqref{eq:frd}, then \[ D_0^- (\Lambda) \geq 1.\]
\end{lemma}
\begin{proof}
 If $\{\gf \}_{\lambda \in \Lambda}$ is a frame for $\mathcal{H}$ with canonical dual frame $\{\df\}_{\lambda \in \Lambda}$, then $\langle \gf, \df \rangle \leq 1$ for all $\lambda \in \Lambda$, so that $M^+ (\{\gf \}_{\lambda \in \Lambda}) \leq 1$.
\end{proof}

\begin{lemma} \label{lem:density_rkhs2}
Let $\{\gf \}_{\lambda \in \Lambda}$ be a frame for $\mathcal{H}$ satisfying the frame measure/density formulae \eqref{eq:frd} and let   $0<A \leq B < \infty$ be frame bounds for $\{ \gf \}_{\lambda \in \Lambda}$. Suppose that
\[
 0 < c \leq \| g_{\lambda} \|^2 \leq C < \infty.
\]
Then
\[
\frac{A}{C} \leq D_0^- (\Lambda) \leq  D_0^+ (\Lambda)  \leq \frac{B}{c}.
\]
\end{lemma}
\begin{proof}
If $S_{\Lambda} : \mathcal{H} \to \mathcal{H}$ denotes the frame operator of $\{\gf\}_{\lambda \in \Lambda}$, then $A \mathbf{I}_{\mathcal{H}} \leq S_{\Lambda} \leq B \mathbf{I}_{\mathcal{H}}$, and hence
\[
 \langle \gf, \df \rangle = \langle \gf , S_{\Lambda}^{-1} \gf \rangle \leq \frac{1}{A} \| \gf \|^2 \leq \frac{C}{A}, \quad \lambda \in \Lambda.
\]
It follows therefore that \[  1 / D_0^- (\Lambda) = M^+ (\{\gf \}_{\lambda \in \Lambda}) \leq C/A.\]
Similarly, using that $\langle \gf, \df \rangle \geq c / B$, it follows that $ 1 / D_0^+ (\Lambda) \geq c / B$.
\end{proof}

\section{Frames near the critical density} \label{sec:frame_critical}
In this section, we prove the existence of frames of reproducing kernels near the critical density. Throughout, we denote by $\mathcal{H} \subseteq L^2 (X, \mu)$ a reproducing kernel Hilbert space satisfying the assumptions outlined in \Cref{sec:repro_measure}.

We start with stating the following lemma, which is an adaption of \cite[Proposition 2]{balan2006density} to the present setting; see also \cite[Theorem 4.5]{caspers2023overcompleteness}. 

\begin{lemma} \label{lem:near_critical}
Let $\{\gf \}_{\lambda \in \Lambda}$ be a frame for $\mathcal{H}$ with 
$0 < D_0^-(\Lambda) \leq D^+_0(\Lambda) < \infty$. 
Denote by $\{h_{\lambda} \}_{\lambda \in \Lambda}$ the canonical dual frame of $\{\gf \}_{\lambda \in \Lambda}$. For $0 < \alpha < 1$, define the index set $\Lambda_{\alpha} := \{\lambda \in \Lambda : \langle g_{\lambda}, \df \rangle \leq \alpha \}$. Then
\[
 \frac{\alpha - M^+ (\{\gf \}_{\lambda \in \Lambda})}{\alpha} D_0^- (\Lambda) \leq D_0^- (\Lambda_{\alpha}).
\]
\end{lemma}
\begin{proof}
If $M^+ (\{\gf \}_{\lambda \in \Lambda}) \geq \alpha$, then the conclusion is trivial, so we may assume that we have $M^+ (\{\gf \}_{\lambda \in \Lambda}) < \alpha$. Let $\varepsilon > 0$ be such that $M^+ (\{\gf \}_{\lambda \in \Lambda}) + \varepsilon < \alpha$. Choose an increasing sequence $(r_n)_{n \in \mathbb{N}}$  in $(0, \infty)$ and points $x_n \in X$ such that
\[
D^-_0 (\Lambda_{\alpha}) = \lim_{n \to \infty} \frac{\#(\Lambda_{\alpha} \cap B_{r_n}(x_n))}{\int_{B_{r_n}(x_n)} k(y,y) \; d\mu(y)},
\]
and define 
\[
\widetilde{M}^+(\{g_{\lambda} \}_{\lambda \in \Lambda} ) := \limsup_{n \to \infty} \frac{1}{\#(\Lambda \cap B_{r_n} (x_n))} \sum_{\lambda \in \Lambda \cap B_{r_n}(x_n)} \langle \gf, \df \rangle.
\]
Note that $0 \leq \widetilde{M}^+(\{g_{\lambda} \}_{\lambda \in \Lambda} ) \leq M^+(\{g_{\lambda} \}_{\lambda \in \Lambda} ) \leq 1$. Choose a subsequence $(r_{n_k})_{k \in \mathbb{N}}$ of $(r_n)_{n \in \mathbb{N}}$ and fix $N \in \mathbb{N}$ so large that
\[
 \bigg| \frac{1}{\#(\Lambda \cap B_{r_{n_k}} (x_{n_k}))} \sum_{\lambda \in \Lambda \cap B_{r_{n_k}}(x_{n_k})} \langle \gf, \df \rangle - \widetilde{M}^+ (\{\gf\}_{\lambda \in \Lambda} )\bigg| < \varepsilon
\]
for all $k \geq N$. 
Then, using the definition of $\Lambda_{\alpha}$, a direct calculation gives for $k \geq N$,
\begin{align*}
 \widetilde{M}^+ (\{\gf\}_{\lambda \in \Lambda} ) + \varepsilon  &\geq \frac{1}{\#(\Lambda \cap B_{r_{n_k}} (x_{n_k}))} \sum_{\lambda \in \Lambda \cap B_{r_{n_k}} (x_{n_k})} \langle \gf, h_{\lambda} \rangle \\
 &= \frac{1}{\#(\Lambda \cap B_{r_{n_k}}(x_{n_k}))} \bigg( \sum_{\lambda \in \Lambda_{\alpha} \cap B_{r_{n_k}} (x_{n_k})}  \langle \gf, \df \rangle + \sum_{\lambda \in (\Lambda \setminus \Lambda_{\alpha} )\cap B_{r_{n_k}} (x_{n_k})} \langle \gf, \df \rangle \bigg) \\
 &\geq \frac{1}{\#(\Lambda \cap B_{r_{n_k}}(x_{n_k}))} \bigg( 0 + \sum_{\lambda \in (\Lambda \setminus \Lambda_{\alpha}) \cap B_{r_{n_k}} (x_{n_k})} \alpha \bigg) \\
 & = \alpha \frac{\# (\Lambda \cap B_{r_{n_k}} (x_{n_k})) - \# (\Lambda_{\alpha} \cap B_{r_{n_k}}(x_{n_k}))}{\#(\Lambda \cap B_{r_{n_k}} (x_{n_k}))},
\end{align*}
and thus
\begin{align*}
 \frac{\#(\Lambda_{\alpha} \cap B_{r_{n_k}}(x_{n_k}))}{\int_{B_{r_{n_k}}(x_{n_k})} k(y,y) d\mu(y)} &\geq \bigg( 1 - \frac{\widetilde{M}^+ (\{\gf \}_{\lambda \in \Lambda} ) + \varepsilon}{\alpha} \bigg) \frac{\#(\Lambda \cap B_{r_{n_k}} (x_{n_k}))}{\int_{B_{n_k}(x_{n_k})} k(y,y) d\mu(y)} \\
 &\geq \bigg( 1 - \frac{M^+ (\{\gf \}_{\lambda \in \Lambda} ) + \varepsilon}{\alpha} \bigg) D^-_0 (\Lambda).
\end{align*}
Since $D_0^-(\Lambda) > 0$, $M^+(\{\gf \}_{\lambda \in \Lambda}) + \varepsilon < \alpha$ and $\varepsilon > 0$ may be chosen arbitrary small, this yields
\[
 D_0^-(\Lambda_{\alpha}) \geq \bigg(1 - \frac{M^+(\{\gf \}_{\lambda \in \Lambda}) }{\alpha} \bigg) D_0^- (\Lambda),
\]
as required.
\end{proof}

The following result is our main ingredient for constructing frames near the critical density.

\begin{proposition} \label{prop:main}
Suppose $\{g_x\}_{x \in X}$ is a system of vectors in $\mathcal{H} \subseteq L^2(X)$ such that
\[
0 < c \leq \| g_x \|^2 \leq C < \infty \quad \text{for all} \quad x \in X,
\]
and that satisfies the frame measure/density formulae \eqref{eq:frd}.

Suppose that $\Lambda \subseteq X$ is a countable set such that $\{\gf \}_{\lambda \in \Lambda}$ is a frame for $\mathcal{H}$ with frame bounds $0 < A \leq B < \infty$. Then, for every $\varepsilon \in (0,1]$, there exists  a constant $\delta = \delta(\varepsilon, c/B) < 1$ with the following property:
If $D^-_0 (\Lambda) > 1+ \varepsilon$, then for all $\eta > 0$ there exists a set $\Lambda' \subseteq \Lambda$ satisfying
\[
D^-_0(\Lambda') > 0, \quad D_0^+(\Lambda') - D^-_0 (\Lambda') < \eta \quad \text{and} \quad D^+_0 (\Lambda \setminus \Lambda') \leq \delta D^+_0(\Lambda)
\]
 and such that $\{ g_{\lambda} \}_{\lambda \in \Lambda \setminus \Lambda'}$ is a frame for $\mathcal{H}$ with frame bounds $A \varepsilon / (\varepsilon + 4)$ and $B$.
\end{proposition}
\begin{proof}
Denote by $\{\df\}_{\lambda \in \Lambda}$ the canonical dual frame of $\{\gf\}_{\lambda \in \Lambda}$.
If $1+\varepsilon < D_0^-(\Lambda)$, then the frame measure/density formulae \eqref{eq:frd} yields
\[
 M^+ (\{\gf\}_{\lambda \in \Lambda})  < \frac{1}{1+\varepsilon}.
\]
Set $\alpha := 1/(1+\varepsilon/2)$ and  $\Lambda_{\alpha} := \{\lambda \in \Lambda : \langle \gf, \df \rangle \leq \alpha \}$. Note that \Cref{lem:density_rkhs2} gives, in particular, $0 < D_0^-(\Lambda) \leq D_0^+ (\Lambda) < \infty$. Hence, \Cref{lem:near_critical} is applicable, and yields
\begin{align*}
 D_0^-(\Lambda_{\alpha})
 &\geq  \; \frac{\alpha - M^+(\{\gf\}_{\lambda \in \Lambda})}{\alpha} D_0^-(\Lambda) \\
 &\geq \frac{1/(1+\varepsilon/2) - 1/(1+\varepsilon)}{1+\varepsilon/2} \; D_0^-(\Lambda) \\
 &= \frac{\varepsilon/2}{(1+\varepsilon/2)^2 (1+\varepsilon)} D_0^-(\Lambda) \\
 &\geq \frac{\varepsilon}{9} D_0^- (\Lambda). \numberthis \label{eq:positive_density}
\end{align*}
Let $S_{\Lambda} : \mathcal{H} \to \mathcal{H}$ be the frame operator of $\{\gf\}_{\lambda \in \Lambda}$. Note that $\{S_{\Lambda}^{-1/2} g_{\lambda} \}_{\lambda \in \Lambda}$ is a Parseval frame for $\mathcal{H}$, hence a Bessel sequence with upper bound $1$, and satisfies
\[
 \big\| S_{\Lambda}^{-1/2} \gf \big\|^2 = \langle \gf, \df \rangle \leq \alpha, \quad \lambda \in \Lambda_{\alpha}.
\]
We aim to apply \Cref{thm:selector} to the Bessel sequence $\{S_{\Lambda}^{-1/2} g_{\lambda} \}_{\lambda \in \Lambda_{\alpha}}$. For this, we choose $r \in \mathbb{N}$ such that
\begin{align} \label{eq:choosingr}
 \bigg(\frac{1}{\sqrt{r}} + \sqrt{\alpha} \bigg)^2 < \frac{1}{1+\varepsilon/4} < 1.
\end{align}
The remainder of the proof is split intro three steps.
\\~\\
\textbf{Step 1.}  In this step, we construct a collection $\{J_k\}_{k \in K}$ of disjoint subsets of $\Lambda_{\alpha}$ with $\# J_k \geq r$ for all $k \in K$. Let $\mu_0$ be the measure given by $d\mu_0(y)=k(y,y) d\mu(y)$, $y\in X$. Note that by \eqref{eq:dc} the measure $\mu_0$ satisfies the standing assumptions \eqref{eq:NDB}--\eqref{eq:doubling}.
Since $D_0^-(\Lambda_{\alpha}) \geq \frac{\varepsilon}{9} D_0^-(\Lambda)$, for all $\varepsilon_1 > 0$, there exists $R_{\varepsilon_1} > 0$ such that, for all $R \geq R_{\varepsilon_1}$ and $x \in X$,
\begin{align*}
 \#(\Lambda_{\alpha} \cap B_R(x)) \geq  \bigg( \frac{\varepsilon}{9} D^- (\Lambda) - \varepsilon_1 \bigg) \mu_0(B_R(x)).
\end{align*}
We fix a small $\varepsilon_1>0$ and some $R \geq R_{\varepsilon_1}$ sufficiently large, so that
 \begin{align} \label{eq:positivecardinality}
 \#(\Lambda_{\alpha} \cap B_R(x)) \geq  \bigg(  \frac{\varepsilon}{9} D^- (\Lambda) - \varepsilon_1 \bigg) \mu_0(B_R(x)) \geq r,
\end{align}
see \Cref{lem:uniform_radii}.

Let $(B_R(y))_{y \in Y}$ be a maximal family of disjoint balls in $X$, so that
\[
 \bigcup_{y \in Y} B_{2R} (y) = X
\]
by maximality.
Choose a partition $(X_y)_{y \in Y}$ of $X$ such that
\begin{align} \label{eq:partition0}
 B_R(y) \subseteq X_y \subseteq B_{2R}(y), \quad y \in Y.
\end{align}
 For such a partition, it follows from \eqref{eq:positivecardinality} that
\[
 \#(\Lambda_{\alpha} \cap X_y) \geq \#(\Lambda_{\alpha} \cap B_R(y) ) \geq r, \quad y \in Y.
\]
As such, we can partition each of the sets $\Lambda_{\alpha} \cap X_y$ into sets $\{J_{y, \ell} \}_{\ell \in L}$ with cardinality $r\leq \# J_{y, \ell} < 2r$ for  $\ell \in L$ and $y \in Y$. Hence, if $K := \{ (y, \ell) \in Y \times L: \Lambda_{\alpha} \cap X_y \neq \emptyset\}$, then the sets  $\{J_k\}_{k \in K}$ yield a partition of $\Lambda_{\alpha}$ satisfying for all $k\in K$,
\begin{align} \label{eq:partition}
 r \leq \# J_k < 2r \quad \text{and} \quad J_k \subseteq X_y \quad \text{for some} \quad y \in Y.
\end{align}
This yields the desired partition.
\\~\\
\textbf{Step 2.} Applying \Cref{thm:selector} to the Bessel sequence $\{S_{\Lambda}^{-1/2} g_{\lambda} \}_{\lambda \in \Lambda_{\alpha}}$ and the partition $\{J_k\}_{k \in K}$ of $\Lambda_{\alpha}$ constructed in Step 1, we obtain a selector $J \subseteq \bigcup_{k \in K} J_k$ satisfying $\# (J \cap J_k) = 1$ for all $k \in K$ and such that $\{S_{\Lambda}^{-1/2} g_{\lambda} \}_{\lambda \in J}$ is a Bessel system with bound $\big(\frac{1}{\sqrt{r}} + \sqrt{\alpha} \big)^2$. In this step, we will show that
\begin{align} \label{eq:densityselector}
\frac{1}{2r}  D_0^-(\Lambda_{\alpha}) \leq D^-_0 (J) \leq \frac{1}{r} D^-_0 (\Lambda_{\alpha}).
\end{align}

In order to show \eqref{eq:densityselector}, we start by showing that
\begin{align} \label{eq:densityselector0}
 \frac{1}{2r} \# (X_y \cap \Lambda_{\alpha}) \leq \# (X_y \cap J) \leq \frac{1}{r} \# (X_y \cap \Lambda_{\alpha} ), \quad y \in Y,
\end{align}
where $(X_y)_{y \in Y}$ is the partition of $X$ satisfying \eqref{eq:partition0} and \eqref{eq:partition}.
For $y \in Y$, define the index set $\widetilde{K}_y := \{ k \in K : J_k \subseteq X_y\}$, so that
\begin{align} \label{eq:discrete_partition}
 X_y \cap \Lambda_{\alpha} = \bigcupdot_{k \in \widetilde{K}_y} J_k.
\end{align}
Since $\#J_k \geq r$ and $\#(J_k \cap J) = 1$ for  $k \in K$, necessarily
$
 r \#(J_k \cap J) \leq \# J_k
$ for $k \in K$.
This, together with the fact that $\{J_k \cap J\}_{k \in \widetilde{K}}$ is a partition of $X_y \cap J \subseteq X_y \cap \Lambda_{\alpha}$ and the identity \eqref{eq:discrete_partition}, yields
\[
 r \# (X_y \cap J) = \sum_{k \in \widetilde{K}_y} r \# (J_k \cap J) \leq \sum_{k \in \widetilde{K}_y} \# J_k = \# (X_y \cap \Lambda_{\alpha}),
\]
which shows the upper bound in \eqref{eq:densityselector0}. For the lower bound in \eqref{eq:densityselector0}, note that we have $2r \#(J_k \cap J) > \#J_k$ for $k \in K$. Similar arguments as above therefore yield that
\[
 2r \# (X_y \cap J) =  \sum_{k \in \widetilde{K}_y} 2r \# (J_k \cap J) > \sum_{k \in \widetilde{K}_y} \# J_k = \# (X_y \cap \Lambda_{\alpha}),
\]
showing the lower bound in \eqref{eq:densityselector0}.

Finally, for showing \eqref{eq:densityselector}, choose $R' > 4R$. On one hand, for $x \in X$, it follows from \eqref{eq:densityselector0} that
\begin{align} \label{eq:cardinality1}
 \frac{1}{2r} \# \bigg( \bigg( \bigcup_{y \in Y: d(y,x) < R' + 2R} X_y \bigg) \cap \Lambda_{\alpha} \bigg) \leq \# \bigg( \bigg( \bigcup_{y \in Y: d(y,x) < R' + 2R} X_y \bigg) \cap J \bigg)
\end{align}
and
\begin{align} \label{eq:cardinality2}
 \# \bigg( \bigg( \bigcup_{y \in Y: d(y,x) < R' - 2R} X_y \bigg) \cap J \bigg)
 \leq \frac{1}{r} \# \bigg( \bigg( \bigcup_{y \in Y: d(y,x) < R' - 2R} X_y \bigg) \cap \Lambda_{\alpha} \bigg).
\end{align}
On the other hand, it follows from \eqref{eq:partition0} and the triangle inequality that
\begin{align} \label{eq:sandwich1}
 B_{R'}(x) \subseteq \bigcup_{y \in Y : d(y,x) < R' + 2R} X_y \subseteq B_{R' + 4R}(x)
\end{align}
and
\begin{align} \label{eq:sandwich2}
 B_{R'- 4R}(x) \subseteq \bigcup_{y \in Y : d(y,x) < R' - 2R} X_y \subseteq B_{R'}(x).
\end{align}
Combining \eqref{eq:cardinality1} and \eqref{eq:sandwich1} yields
\[
 \#(B_{R'+4R}(x) \cap J) \geq \frac{1}{2r} \# (B_{R'}(x) \cap \Lambda_{\alpha})
\]
whereas combining \eqref{eq:cardinality2} and \eqref{eq:sandwich2} gives
\[
 \#(B_{R' - 4R}(x) \cap J) \leq \frac{1}{r} \# (B_{R'}(x) \cap \Lambda_{\alpha}).
\]
Therefore,
\[
 \frac{\#(B_{R'+4R}(x) \cap J)}{\mu_0(B_{R' + 4R}(x))} \geq \frac{1}{2r} \frac{\# (B_{R'}(x) \cap \Lambda_{\alpha})}{\mu_0(B_{R'}(x))} \frac{\mu_0(B_{R'}(x))}{\mu_0(B_{R' + 4R}(x))}
\]
and \[
 \frac{\#(B_{R'-4R}(x) \cap J)}{\mu_0(B_{R' - 4R}(x))} \leq \frac{1}{r} \frac{\# (B_{R'}(x) \cap \Lambda_{\alpha})}{\mu_0(B_{R'}(x))} \frac{\mu_0(B_{R'}(x))}{\mu_0(B_{R' - 4R}(x))}.
\]
This readily implies that $D_0^- (J) \geq \frac{1}{2r} D_0^- (\Lambda_{\alpha})$ and $D_0^-(J) \leq \frac{1}{r} D_0^- (\Lambda_{\alpha})$ since
\[
\lim_{R' \to \infty} \sup_{x \in X} \frac{\mu_0(B_{R'} (x))}{\mu_0(B_{R'+4R} (x))} = 1 = \lim_{R' \to \infty} \sup_{x \in X} \frac{\mu_0(B_{R'} (x))}{\mu_0(B_{R'-4R} (x))}
\]
by \Cref{lem:wad}. 
\\~\\
\textbf{Step 3.} We start by showing that $\{g_{\lambda} \}_{\lambda \in \Lambda \setminus J}$ is a frame for $\mathcal{H}$ with frame bounds $A \varepsilon / (\varepsilon + 4)$ and $B$.
Since \Cref{thm:selector} yields that $\{S_{\Lambda}^{-1/2} g_{\lambda} \}_{\lambda \in J}$ is a Bessel sequence with upper frame bound $(1/\sqrt{r} + \sqrt{\alpha} )^2 < 1/(1+\varepsilon/4)$ (cf. \Cref{eq:choosingr}), the system $\{S_{\Lambda}^{-1/2} g_{\lambda} \}_{\lambda  \in \Lambda \setminus J}$ satisfies
\begin{align*}
\| f \|^2 &\geq  \sum_{\lambda \in \Lambda \setminus J} |\langle f, S_{\Lambda}^{-1/2} g_{\lambda} \rangle |^2 = \sum_{\lambda \in \Lambda} |\langle f, S_{\Lambda}^{-1/2} g_{\lambda} \rangle |^2 - \sum_{\lambda \in J} |\langle f, S_{\Lambda}^{-1/2} g_{\lambda} \rangle |^2 \\
&\geq \bigg(1 - \frac{1}{1+\varepsilon/4} \bigg) \| f \|^2 = \frac{\varepsilon}{\varepsilon + 4} \| f \|^2,
\end{align*}
for all $f \in \mathcal{H}$. Thus, $\{S_{\Lambda}^{-1/2} g_{\lambda} \}_{\lambda \in \Lambda \setminus J}$ forms a frame for $\mathcal{H}$ with lower bound $ \varepsilon / (\varepsilon + 4)$ and upper bound $1$.
Denote by $S_{\Lambda \setminus J} : \mathcal{H} \to \mathcal{H}$  the frame operator of $\{\gf \}_{\lambda \in \Lambda \setminus J}$, and note that
\[
S_{\Lambda \setminus J} = \sum_{\lambda \in \Lambda \setminus J} \langle \cdot , g_{\lambda} \rangle g_{\lambda} = S^{1/2}_{\Lambda} \bigg( \sum_{\lambda \in \Lambda \setminus J} \langle \cdot , S^{-1/2}_{\Lambda} \gf \rangle S^{-1/2}_{\Lambda} \gf \bigg) S^{1/2}_{\Lambda}.
\]
and 
\[
\frac{\varepsilon}{\varepsilon+4} \mathbf{I} \leq \sum_{\lambda \in \Lambda \setminus J} \langle \cdot , S^{-1/2}_{\Lambda} \gf \rangle S^{-1/2}_{\Lambda} \gf \leq \mathbf{I}.
\]
Combining this with $A \mathbf{I} \leq S_{\Lambda} \leq B \mathbf{I}$ yields
\begin{align*}
A \frac{ \varepsilon}{\varepsilon + 4} \textbf{I} &\leq \frac{ \varepsilon}{\varepsilon + 4} S_{\Lambda} = S^{1/2}_{\Lambda} \bigg( \frac{\varepsilon}{\varepsilon + 4}  \textbf{I} \bigg) S^{1/2}_{\Lambda} \leq S_{\Lambda \setminus J} \leq S_{\Lambda} \leq B  \textbf{I},
\end{align*}
which shows that $\{\gf \}_{\lambda \in \Lambda \setminus J}$ is a frame for $\mathcal{H}$ with frame bounds $A \varepsilon / (\varepsilon + 4)$ and $B $.

Without loss of generality, we may assume that $\eta < \min\{ D^-_0(J),  \frac{\varepsilon}{36r} D^-_0 (\Lambda) \}$.
For such a choice of $\eta$, we apply \Cref{sub} to obtain a set $\Lambda' \subseteq J$ such that
\[
 D^-_0(J) - \eta \leq D^-_0(\Lambda') \leq D^+_0(\Lambda') \leq D^-_0 (J).
\]
Note that $\{\gf \}_{\lambda \in \Lambda \setminus \Lambda'}$ forms still a frame for $\mathcal{H}$ with frame bounds $A \varepsilon / (\varepsilon + 4)$ and $B $, that $D^-_0(\Lambda') > 0$ and $D^+_0 (\Lambda') - D^-_0(\Lambda') < \eta$. Therefore, it remains therefore to estimate the density of $\Lambda \setminus \Lambda'$. For this, first observe that
\[ D_0^+(\Lambda \setminus \Lambda') \leq D_0^+ (\Lambda) - D_0^- (\Lambda')\]
and
\begin{align*}
 D_0^-(\Lambda') = D_0^- (J) - \eta \geq \frac{1}{2r}D_0^- (\Lambda_{\alpha}) - \eta \geq \frac{\varepsilon}{18r} D_0^-(\Lambda) - \eta &\geq  \frac{\varepsilon}{36r} D_0^-(\Lambda).
\end{align*}
by \eqref{eq:positive_density}, \eqref{eq:densityselector} and the choice of $\eta$. Therefore,
\begin{align*}
 0< D_0^+(\Lambda \setminus \Lambda') &\leq  \bigg(1 -   \frac{\varepsilon}{36r}\frac{D_0^-(\Lambda)}{D_0^+(\Lambda)} \bigg) D_0^+ (\Lambda) \\
 &\leq  \bigg(1 -   \frac{\varepsilon}{36r}  \frac{1}{D_0^+(\Lambda)} \bigg) D_0^+ (\Lambda) \\
 &\leq \bigg(1 -   \frac{\varepsilon}{36r}  \frac{c}{B} \bigg) D_0^+ (\Lambda),
\end{align*}
where the last two steps used that $D_0^-(\Lambda) \geq 1$ and $1 /D_0^+(\Lambda) \geq c/B$; see \Cref{lem:density_rkhs} and \Cref{lem:density_rkhs2}. Setting
\[
\delta := \bigg(1 -   \frac{\varepsilon}{36r}  \frac{c}{B} \bigg) \in (0,1)
\]
gives the desired result.
\end{proof}

\begin{theorem} \label{thm:main0}
Let $\{g_x\}_{x \in X}$ be a system of vectors in $\mathcal{H} \subseteq L^2(X)$ such that
\[
0 < c \leq \| g_x \|^2 \leq C < \infty \quad \text{for all} \quad x \in X,
\]
and that satisfies the frame measure/density formulae \eqref{eq:frd}. Let $\varepsilon > 0$.

If $\Lambda \subseteq X$ is a countable set such that $\{\gf \}_{\lambda \in \Lambda}$ is a frame for $\mathcal{H}$ with frame bounds $0 < A \leq B < \infty$,
then there exists $\Gamma \subseteq \Lambda$ satisfying
$D_0^-(\Gamma) \leq 1+\varepsilon$ and such that $\{ g_{\gamma} \}_{\gamma \in \Gamma}$ is a frame for $\mathcal{H}$ with frame bounds $c' A$ and $B$ for some constant $c' = c'(\varepsilon, c/B) > 0$.

In addition, if $\varepsilon_0 := D^+_0(\Lambda) - D^-_0 (\Lambda)$, then $\Gamma \subseteq \Lambda$ can be chosen to satisfy $D_0^+(\Gamma) \leq 1+\varepsilon_0 + \varepsilon$.
\end{theorem}
\begin{proof}
Without loss of generality, we will assume throughout the proof that $\varepsilon \in (0,1]$ and  that $D_0^- (\Lambda) > 1+\varepsilon/2$.

First, let $\eta > 0$ be arbitrary. Since $D_0^- (\Lambda) > 1+\varepsilon$/2, an application of \Cref{prop:main} yields a positive constant $\delta = \delta(\varepsilon, c/ B) < 1$ and a subset $\Lambda_1 \subseteq \Lambda$ satisfying $D_0^+ (\Lambda_1) \leq \delta D_0^+ (\Lambda)$ and $D^+_0 (\Lambda_1) - D^-_0 (\Lambda_1) < \eta$ and such that $\{\gf \}_{\lambda \in \Lambda_1}$ is a frame for $\mathcal{H}$ with frame bounds $A \varepsilon / (\varepsilon + 8)$ and $B$.
Applying \Cref{prop:main} recursively, we obtain a sequence of sets
\begin{align} \label{eq:sequence_sets}
\Lambda \supseteq \Lambda_1 \supseteq ... \supseteq \Lambda_n
\end{align}
satisfying $D^+_0 (\Lambda_i) - D^-_0(\Lambda_i) < \eta$ for $i = 1, ..., n$, and
\[
D_0^+ (\Lambda_n) \leq \delta D_0^+ (\Lambda_{n-1}) \leq ... \leq \delta^n D_0^+ (\Lambda),
\]
 and such that $\{\gf \}_{\lambda \in \Lambda_n}$ is a frame for $\mathcal{H}$ with frame bounds $A (\varepsilon / (\varepsilon + 8))^n$ and $B$, as long as $D^-_0(\Lambda_n) > 1 + \varepsilon / 2$.
Note that the number of steps $n$ can be at most $N = N (\varepsilon, c/B) \in \mathbb{N}$, where $N \in \mathbb{N}$ is the smallest integer such that
\[
\delta^{N} \frac{B}{c} \leq 1+\varepsilon/2 .
\]
Indeed, since $D_0^+ (\Lambda)  \leq B / c$ by \Cref{lem:density_rkhs2}, it follows that
\[
D_0^-(\Lambda_{N}) \leq D^+_0(\Lambda_{N}) \leq \delta^{N} D_0^+ (\Lambda) \leq \delta^{N} \frac{B}{c} \leq 1+\varepsilon /2,
\]
Let $n_0 \leq N$ be the smallest integer such that $D^- (\Lambda_{n_0}) \leq 1 + \varepsilon / 2$. Then $\{g_{\lambda} \}_{\lambda \in \Lambda_{n_0}}$ is a frame for $\mathcal{H}$ with $D^- (\Lambda_{n_0}) \leq 1 + \varepsilon / 2$ and frame bounds $A (\varepsilon /(\varepsilon + 8))^N$ and $B$. Hence, setting $\Gamma := \Lambda_{n_0}$ and
$c' := (\varepsilon/(\varepsilon + 8))^N$ yields the first claim.

It remains to prove the additional part on the upper density. For this, we choose $\eta := \frac{\varepsilon}{2} \frac{1}{N}$ and set $\Lambda_0 := \Lambda$.
For $i = 1, ..., n_0$, note that each set in \eqref{eq:sequence_sets} is of the form $\Lambda_i = \Lambda_{i - 1} \setminus \Lambda'_{i - 1}$ for some $\Lambda'_{i-1} \subseteq \Lambda_{i-1}$. Therefore,
\[
 D^+_0(\Lambda_i) \leq D^+_0 (\Lambda_{i-1}) - D^-_0 (\Lambda'_{i - 1}) \quad \text{and} \quad D^-_0(\Lambda_i) \leq D^-_0 (\Lambda_{i-1}) - D^+_0 (\Lambda'_{i - 1}),
\]
so that
\[
 D^+_0 (\Lambda_i) - D_0^- (\Lambda_i) \leq D^+_0 (\Lambda_{i -1}) - D^-_0 (\Lambda_{i - 1}) + \eta.
\]
for $i = 1, ..., n_0$. This shows that the set $\Lambda_{n_0}$ satisfies
\[
 D^+_0 (\Lambda_{n_0}) - D^-_0 (\Lambda_{n_0}) \leq D^+_0 (\Lambda) - D^-_0 (\Lambda) + n_0 \eta = \varepsilon_0 + \frac{\varepsilon}{2} \frac{n_0}{N}.
\]
Since $D^-_0 (\Lambda_{n_0}) \leq 1+ \varepsilon /2$ and $n_0 \leq N$ this implies
\[
 D^+_0 (\Lambda_{n_0}) \leq 1+\varepsilon + \varepsilon_0,
\]
which finishes the proof.
\end{proof}

The following theorem implies, in particular, the existence of frames of reproducing kernels near the critical density:

\begin{theorem} \label{thm:main}
Suppose $\{g_x\}_{x \in X}$ is a frame for $\mathcal{H} \subseteq L^2(X)$ with frame bounds $0<A\leq B < \infty$ satisfying
\[
0 < c \leq \| g_x \|^2 \leq C < \infty \quad \text{for all} \quad x \in X,
\]
and such that frame measure/density formulae \eqref{eq:frd} holds.
Then, for each $\varepsilon > 0$, there exists countable $\Lambda \subseteq X$ satisfying $D_0^+(\Lambda) \leq 1+\varepsilon$  such that $\{ g_{\lambda} \}_{\lambda \in \Lambda}$ is a frame for $\mathcal{H}$.

In addition, if $\{g_x \}_{x \in X}$ is a Parseval frame for $\mathcal{H}$ satisfying $C = \| g _x \|^2$ for all $x \in X$, then the frame bounds of the frame $\{ g_{\lambda} \}_{\lambda \in \Lambda}$ can be chosen to be of the form $c_1(\varepsilon) C$ and $c_2 C$ for constants $c_1(\varepsilon), c_2 > 0$ with $c_1 (\varepsilon)$  depending only on $\varepsilon$.
\end{theorem}
\begin{proof}
 Under the assumptions, an application of \Cref{thm:discretization} yields an absolute constant $c_0 > 0$ and a countable set $\Lambda_1 \subseteq X$ such that $\{ g_{\lambda} \}_{\lambda \in \Lambda_1}$ is a frame for $\mathcal{H}$ with frame bounds $c_0 C (A-\eta)/\eta^2$ and $2 c_0C (B+ \eta)/\eta^2$, where $\eta := A/2$. Hence,
 it follows from \Cref{lem:density_rkhs2} that
 \[
 \frac{c_0  (A-\eta)}{\eta^{2}}   \leq D_0^-(\Lambda_1) \leq D_0^+ (\Lambda_1) \leq  \frac{2c_0 C (B+\eta)}{c \eta^{2}}  ,
 \]
 and thus
 \begin{align*}
  D^+_0 (\Lambda_1) - D^-_0 (\Lambda_1) \leq   c_0 \bigg(2 \frac{C}{c} \frac{B+\eta}{\eta^2} - \frac{A-\eta}{\eta^2} \bigg).
 \end{align*}
 Next, fix $0< \delta = \delta(C/c, A, B) < 1$ such that
 \[
 \frac{A-\delta }{\delta^2 } \geq 2 \frac{C}{c} \frac{B+\eta}{\eta^2} - \frac{A-\eta}{\eta^2}.
 \]
 Then, by an application of \Cref{thm:discretization} with $\delta$,
 we obtain a set $\Lambda_2$ such that $\{\gf\}_{\lambda \in \Lambda_2}$ is a frame for $\mathcal{H}$ with frame bounds $(A-\delta)c_0 C/\delta^2$ and $2(B+\delta)c_0 C/\delta^2$. By Lemma \ref{lem:density_rkhs2} and the choice of $\delta$, we therefore obtain
 \[
  D^-_0(\Lambda_2) \geq \frac{ (A-\delta) c_0 }{\delta^2} \geq D^+_0 (\Lambda_1) - D_0^- (\Lambda_1).
 \]
Thus, an application of Lemma \ref{sup} yields a set $\Lambda \subseteq X$ satisfying $\Lambda_1 \subseteq \Lambda \subseteq \Lambda_1 \cup \Lambda_2$ and
\[
 \varepsilon_0 := D^+_0(\Lambda) - D^-_0(\Lambda) \leq \frac{\varepsilon}{2}.
\]
Note that $\{\gf  \}_{\lambda \in \Lambda}$ is a frame for $\mathcal{H}$ with lower frame bound $A' = c_0 C (A -\eta)/\eta^{2}$ and upper frame bound $B' = 2 c_0C ((B+\eta)/\eta^2 + (B+ \delta)/\delta^2)$. An application of \Cref{thm:main0} therefore yields that there exists $\Gamma \subseteq \Lambda$ satisfying
\[
 D^+_0 (\Gamma) \leq 1 + \varepsilon_0 + \frac{\varepsilon}{2} \leq 1+\varepsilon
\]
and such that
$\{g_{\gamma} \}_{\gamma \in \Gamma}$ is a frame for $\mathcal{H}$ with frame bounds $c' A'$ and $B'$ for some constant $c' = c' (\varepsilon, c/B') > 0$. This proves the first part of the theorem.

For the additional part, note that if $A=B=1$ and $c=C$, then $\eta = 1/2$, and thus we can choose $\delta=1/4$, so that $c' = c'(\varepsilon)$ only depends on $\varepsilon$. Therefore, setting $c_1(\varepsilon) = 2 c_0 c'(\varepsilon)$ and $c_2 = 2 c_0 ((1+\eta)/\eta^2 + (1+ \delta)/\delta^2)=52c_0$ completes the proof.
\end{proof}

\section{Examples} \label{sec:examples}
This section discusses the results obtained in \Cref{sec:frame_critical} for several particular examples of frames and reproducing kernel Hilbert spaces. In particular, we show that these examples satisfy the standing assumptions and localization assumptions introduced in Section \ref{sec:repro_measure}, which allows to apply the results of \Cref{sec:frame_critical} directly to these examples.

\subsection{Exponential frames} \label{sec:exponential}
Let $\Omega \subseteq \mathbb{R}^d$ be a measurable set of finite measure and denote by $\PW_{\Omega}$ the Paley-Wiener space consisting of functions $f \in L^2 (\mathbb{R}^d)$ whose Fourier transform
\[
\widehat{f} (\xi) = \int_{\mathbb{R}^d} f(t) e^{2\pi i \xi \cdot t} \; dt
\]
has essential support inside $\Omega$. 
 For $x \in \mathbb{R}^d$, define the exponential function $e_x$ by
\[
e_{x} (t) = e^{2\pi i x \cdot t}, \quad t \in \mathbb{R}^d.
\]
 By Plancherel's theorem, 
the Fourier transform $\mathcal{F} : \PW_{\Omega} \to L^2 (\Omega)$ is an isometry, 
and 
\begin{align} \label{eq:fourier_inversion}
f (x) = \int_{\Omega} \widehat{f}(\xi) e^{2 \pi i x \cdot \xi} \; d\xi = \langle \widehat{f} , e_{-x} \rangle_{L^2 (\Omega)} = \langle f, \mathcal{F}^{-1} e_{-x} \mathds{1}_{\Omega} \rangle_{L^2 (\mathbb{R}^d)}, \quad x \in \mathbb{R}^d,
\end{align}
so that $\PW_{\Omega}$ is a reproducing kernel Hilbert space with reproducing kernel given by
\[
k_x (y) = (\mathcal{F}^{-1} e_{-x} \mathds{1}_{\Omega})(y) = (\mathcal{F}^{-1} \mathds{1}_{\Omega}) (y - x)
\] 
for $x,y \in \mathbb{R}^d$. In particular, we have $\| k_x \|^2 = |\Omega|$ for all $x \in \mathbb{R}^d$.

From the identity \eqref{eq:fourier_inversion}, it follows that a set of reproducing kernels $\{k_{\lambda} \}_{\lambda \in \Lambda}$ is a frame for $\PW_{\Omega}$ if and only if the exponential system $\{e_{\lambda} \mathds{1}_{\Omega} \}_{\lambda \in \Lambda}$ is a frame for $L^2 (\Omega)$.

We will apply the results from \Cref{sec:frame_critical} to the Paley-Wiener space $\PW_{\Omega}$ on $\mathbb{R}^d$ equipped with Lebesgue measure and Euclidean metric, which clearly satisfy the assumptions on the metric measure space.
In addition, it is readily verified that $\{k_x \}_{x \in \mathbb{R}^d}$ satisfies the weak localization property \eqref{eq:wl}. 
The homogeneous approximation property \eqref{eq:hap} is satisfied whenever $\Omega$ is bounded, see, e.g., \cite[Lemma 1]{grochenig1996landau} or \Cref{lem:hap_coherent} below.
Using this, we show that the frame measure/density property \eqref{eq:frd} also holds for reproducing kernels in Paley-Wiener spaces with possibly unbounded spectra.

\begin{theorem} \label{thm:frd_exponential}
Let $\Omega \subseteq \mathbb{R}^d$ be a measurable set of finite measure. Suppose that $\Lambda \subseteq \mathbb{R}$ is such that $ \{k_{\lambda} \}_{\lambda \in \Lambda}$ is a frame for $\PW_{\Omega}$. Then
\begin{align*}
M^- (\{k_{\lambda} \}_{\lambda \in \Lambda}) = \frac{|\Omega|}{D^+(\Lambda)} \quad \text{and} \quad M^+ (\{k_{\lambda} \}_{\lambda \in \Lambda}) = \frac{|\Omega|}{D^-(\Lambda)}.
\end{align*}
\end{theorem}
\begin{proof}
Since $\{k_{\lambda} \}_{\lambda \in \Lambda}$ is frame for $\PW_{\Omega}$, we have that $\{e_{\lambda} \mathds{1}_{\Omega} \}_{\lambda \in \Lambda}$ is a frame for $L^2 (\Omega)$. Denote by $S_{\Lambda}$ the frame operator of $\{e_{\lambda} \mathds{1}_{\Omega} \}_{\lambda \in \Lambda}$ on $L^2 (\Omega)$, so that its canonical dual frame is given by $\widetilde{h_{\lambda}} = S^{-1}_{\Lambda} e_{\lambda} \mathds{1}_{\Omega}$ for $\lambda \in \Lambda$.
For $n \in \mathbb{N}$, consider the projection $P_n$ on $L^2 (\Omega)$ given by $P_n f = f \cdot \mathds{1}_{B_n(0)}$.  Then $\{e_{\lambda} \mathds{1}_{\Omega \cap B_n(0)} \}_{\lambda \in \Lambda}$ is a frame for $L^2 (\Omega \cap B_n(0))$ with frame operator given by $S^{(n)}_{\Lambda} = P_n S_{\Lambda} P_n$, and thus $\{e_{\lambda} \mathds{1}_{\Omega \cap B_n(0)} \}_{\lambda \in \Lambda}$ and $\{\widetilde{h_{\lambda}} \mathds{1}_{B_n(0)} \}_{\lambda \in \Lambda}$ are canonical dual frames in $L^2 (\Omega \cap B_n(0))$. Consequently, the reproducing kernels $\{k^{(n)}_{\lambda} \}_{\lambda \in \Lambda}$ form a frame for $\PW_{\Omega \cap B_n (0)}$ with canonical dual frame, say,
 $\{h^{(n)}_{\lambda} \}_{\lambda \in \Lambda}$. Since $\Omega \cap B_n(0)$ is bounded, the system $\{k^{(n)}_{\lambda} \}_{\lambda \in \Lambda}$ satisfies the homogeneous approximation property \eqref{eq:hap}, and thus
\Cref{thm:framemeasure_rkhs} yields that
\begin{align} \label{eq:frbounded}
M^+ (\{ k^{(n)}_{\lambda}\}_{\lambda \in \Lambda}) = \frac{|\Omega \cap B_n (0)|}{D^- (\Lambda)} \quad \text{and} \quad M^- (\{ k^{(n)}_{\lambda}\}_{\lambda \in \Lambda}) = \frac{|\Omega \cap B_n (0)|}{D^+ (\Lambda)}
\end{align}
for $n \in \mathbb{N}$. We start by showing the following claim:
\\~\\
 \textbf{Claim:} The upper and lower frame measure of  $\{k_{\lambda}\}_{\lambda \in \Lambda}$ are given by
 \[
 M^+ (\{k_{\lambda}\}_{\lambda \in \Lambda}) = \lim_{n\to \infty} M^+ (\{ k^{(n)}_{\lambda}\}_{\lambda \in \Lambda}) \quad \text{and} \quad
 M^- (\{k_{\lambda}\}_{\lambda \in \Lambda}) = \lim_{n\to \infty} M^- (\{ k^{(n)}_{\lambda}\}_{\lambda \in \Lambda}),
 \]
respectively.

 To ease notation, we define
\begin{align*}
H^{(r)}_{n}(x) &:= \frac{1}{\# (\Lambda \cap B_r (x))} \sum_{\lambda \in \Lambda \cap B_r (x)} \langle k^{(n)}_{\lambda} , h^{(n)}_{\lambda}  \rangle_{L^2 (\mathbb{R}^d)} 
\end{align*}
and
\begin{align*}
 H^{(r)}(x) &:= \frac{1}{\# (\Lambda \cap B_r (x))} \sum_{\lambda \in \Lambda \cap B_r (x)} \langle k_{\lambda} , h_{\lambda}  \rangle_{L^2 (\mathbb{R}^d)} \\
 \end{align*}
for $x \in \mathbb{R}^d$, $r> 0$ and $n \in \mathbb{N}$. With this notation, the claims read
\begin{align} \label{eq:claim_notation}
\limsup_{r \to \infty} \sup_{x \in \mathbb{R}^d} H^{(r)}(x) = \lim_{n \to \infty} \limsup_{r \to \infty} \sup_{x \in \mathbb{R}^d} H^{(r)}_n (x)
\end{align}
and
\begin{align} \label{eq:claim_notation2}
\liminf_{r \to \infty} \inf_{x \in \mathbb{R}^d} H^{(r)}(x) = \lim_{n \to \infty} \liminf_{r \to \infty} \inf_{x \in \mathbb{R}^d} H^{(r)}_n (x)
\end{align}

For showing \eqref{eq:claim_notation} and \eqref{eq:claim_notation2}, note first that 
\begin{align*}
\langle k_{\lambda} , h_{\lambda} \rangle_{L^2 (\mathbb{R}^d)} = \overline{\langle k_{\lambda} , h_{\lambda} \rangle_{L^2 (\mathbb{R}^d)}} = \langle e_{\lambda} , \overline{\mathcal{F} h_{\lambda}} \rangle_{L^2 (\Omega)}
\end{align*}
and, similarly, $\langle k^{(n)}_{\lambda} , h^{(n)}_{\lambda} \rangle_{L^2 (\mathbb{R}^d)}  = \langle e_{\lambda} , \overline{\mathcal{F} h_{\lambda}} \rangle_{L^2 (\Omega \cap B_n (0))}$. Therefore,
\begin{align*}
\big|\langle k_{\lambda}, h_{\lambda} \rangle_{L^2 (\mathbb{R}^d)} - \langle k_{\lambda}^{(n)}, h^{(n)}_{\lambda} \rangle_{L^2 (\mathbb{R}^d)} \big|
&=|\langle e_{\lambda} \mathds{1}_{\Omega}, \overline{\mathcal{F} h_{\lambda}} \cdot \mathds{1}_{\Omega \setminus (\Omega \cap B_n(0))} \rangle_{L^2 (\Omega)} | \\
&\leq \| e_{\lambda} \mathds{1}_{\Omega} \|_{L^{\infty} (\Omega)} \| \mathcal{F} h_{\lambda} \cdot \mathds{1}_{\Omega \setminus (\Omega \cap B_n(0))} \|_{L^1 (\Omega)} \\
&\leq \| e_{\lambda} \mathds{1}_{\Omega} \|_{L^{\infty} (\Omega)} \| \mathcal{F} h_{\lambda} \|_{L^2 (\Omega)} \| \mathds{1}_{\Omega \setminus (\Omega \cap B_n(0))} \|_{L^2 (\Omega)} \\
&\leq C \| \mathds{1}_{\Omega \setminus (\Omega \cap B_n(0))} \|_{L^2 (\Omega)},
\end{align*}
where we used that $\| e_{\lambda} \mathds{1}_{\Omega} \|_{L^{\infty} (\Omega)} \leq 1$ and $ \| \mathcal{F} h_{\lambda} \|_{L^{2} (\Omega)} = \| h_{\lambda} \|_{L^2 (\mathbb{R}^d)} \leq C$ for all $\lambda \in \Lambda$. 
This yields that
\begin{align*}
\big| H^{(r)} (x) - H_n^{(r)} (x) \big| &\leq  \frac{1}{\# (\Lambda \cap B_r (x))} \sum_{\lambda \in \Lambda \cap B_r (x)} \big|\langle k_{\lambda}, h_{\lambda} \rangle_{L^2 (\mathbb{R}^d)} - \langle k_{\lambda}^{(n)}, h^{(n)}_{\lambda} \rangle_{L^2 (\mathbb{R}^d)} \big|
 \\
&\leq C \| \mathds{1}_{\Omega \setminus (\Omega \cap B_n(0))} \|_{L^2 (\Omega)},
\end{align*}
 and hence
\[
H^{(r)} (x) - C \| \mathds{1}_{\Omega \setminus (\Omega \cap B_n(0))} \|_{L^2 (\Omega)} \leq H_n^{(r)} (x) \leq H^{(r)} (x) + C \| \mathds{1}_{\Omega \setminus (\Omega \cap B_n(0))} \|_{L^2 (\Omega)}
\]
for all $x \in \mathbb{R}^d$, $r > 0$ and $n \in \mathbb{N}$. Taking the supremum (resp. infinum) over $x \in \mathbb{R}^d$, letting $r$ tend to $\infty$ and then $n$ tend to $\infty$, the above inequalities immediately gives \eqref{eq:claim_notation} (resp. \eqref{eq:claim_notation2}), and thus settles the claim.

Using the claim and the identity \eqref{eq:frbounded}, we obtain
\[
M^+ (\{k_{\lambda} \}_{\lambda \in \Lambda}) = \lim_{n \to \infty} M^+ (\{ k^{(n)}_{\lambda}\}_{\lambda \in \Lambda}) = \lim_{n \to \infty} \frac{|\Omega \cap B_n(0)|}{D^- (\Lambda)} = \frac{|\Omega|}{D^- (\Lambda)},
\]
and similarly $M^- (\{k_{\lambda} \}_{\lambda \in \Lambda}) = |\Omega| / D^+(\Lambda)$.
\end{proof}

The following result improves the existence of exponential frames on unbounded sets as shown in \cite[Theorem 1]{nitzan2016exponential} by showing that such frames can be chosen near the critical density. In addition, it extends the main result of \cite{bownik2025on} on the existence of frames near the critical density with ``good'' frame bounds from compact sets to general sets of finite measure.

\begin{theorem} \label{thm:exponential}
 Let $\Omega \subseteq \mathbb{R}^d$ be a set of finite measure. For every $\varepsilon > 0$, there exists a set $\Lambda \subseteq \mathbb{R}^d$ satisfying
 \[
  D^+ (\Lambda) \leq (1+\varepsilon) |\Omega|
 \]
and such that $\{e_{\lambda} \mathds{1}_{\Omega} \}_{\lambda \in \Lambda}$ is a frame for $L^2 (\Omega)$ with frame bounds $c_1 (\varepsilon) |\Omega|$ and $c_2|\Omega|$ for constants $c_1(\varepsilon), c_2  >0$ with $c_1 (\varepsilon)$ only depending on $\varepsilon$.
\end{theorem}
\begin{proof}
The reproducing kernels $\{k_{x} \}_{x \in \mathbb{R}^d}$ form a Parseval frame for $\PW_{\Omega}$ and, by Theorem \ref{thm:frd_exponential},  they satisfy the frame measure/density property \eqref{eq:frd}. Since $\| k_x \|^2 = |\Omega|$ for all $x \in \mathbb{R}^d$, an application of Theorem \ref{thm:main} therefore yields a set $\Lambda \subseteq \mathbb{R}^d$  with $D^+ (\Lambda) \leq (1+\varepsilon) |\Omega|$ and such that $\{k_{\lambda} \}_{\lambda \in \Lambda}$ is frame for $\PW_{\Omega}$ with frame bounds $c_1(\varepsilon) |\Omega|$ and $c_2  |\Omega|$ for constants $c_1(\varepsilon) , c_2 > 0$.
Consequently, the system $\{e_{\lambda} \mathds{1}_{\Omega} \}_{\lambda \in \Lambda}$ is a frame for $L^2 (\Omega)$ with the same frame bounds.
\end{proof}

\subsection{Gabor systems} \label{sec:gabor}
For $(x,\xi) \in \mathbb{R}^{2d}$, define the operator $\pi(x, \xi)$ on $L^2 (\mathbb{R}^d)$ by
\[
\pi(x, \xi) f(t) = e^{2\pi i \xi \cdot t} f(t-x), \quad f \in L^2 (\mathbb{R}^d).
\]
A system of the form $\{\pi(\lambda) g \}_{\lambda \in \Lambda}$ for a set $\Lambda \subseteq \mathbb{R}^{2d}$ is often referred to as a \emph{Gabor system} or \emph{Weyl-Heisenberg system}. In order to obtain an associated reproducing kernel Hilbert space, we consider the map
\[
V_{\phi} : L^2 (\mathbb{R}^{d}) \to L^2 (\mathbb{R}^{2d}), \quad V_{\phi} f(z) = \langle f, \pi(z) \phi \rangle , \quad z=(x,\xi),
\]
associated with a fixed vector $\phi \in L^2 (\mathbb{R}^d)$. We will assume that $\phi$ is a unit vector, so that $V_{\phi} : L^2 (\mathbb{R}^d) \to L^2 (\mathbb{R}^{2d})$ is an isometry by the Moyal identity. In this case,
\[
V_{\phi} f (z) = \langle V_{\phi} f, V_{\phi} (\pi(z) \phi) \rangle = \langle V_{\phi} f, L_z^{\sigma} V_{\phi} \phi \rangle,
\]
where $L_z^{\sigma} F(z') = \sigma(z, z' - z) F(z'-z)$ with $\sigma(z, z') = e^{-2\pi i x \cdot x'}$ for $z = (x,\xi)$, $z' = (x', \xi')$. This shows that $V_{\phi} (L^2 (\mathbb{R}^d))$ is a reproducing kernel Hilbert space with reproducing kernel
\[
k(z,z') = \overline{L_z^{\sigma} V_{\phi} \phi (z')}. 
\]
In particular, $k(z,z) = \| \phi \|^2 = 1$, so that the diagonal condition \eqref{eq:dc} is satisfied.  
For $g \in L^2 (\mathbb{R}^d)$, define the vectors 
\[
g_z := V_{\phi} \pi(z) g, \quad  z \in \mathbb{R}^{2d}.
\] 
Since $V_{\phi} g \in L^2 (\mathbb{R}^{2d})$, it follows that $\{g_z \}_{z \in \mathbb{R}^{2d}}$ satisfies the weak localization property \eqref{eq:wl}. Let $\phi_0(x)= 2^{d/4} e^{-\pi |x|^2}$, $x\in \R^d$, be the Gaussian function.
Then, $V_{\phi_0} (L^2 (\mathbb{R}^{d}))$ consists of entire functions on $\mathbb{C}^d = \mathbb{R}^{2d}$ and
the system $\{g_z \}_{z \in \mathbb{R}^{2d}} = \{ V_{\phi_0} \pi(z) g \}_{z \in \mathbb{R}^{2d}}$ satisfies the homogeneous approximation property \eqref{eq:hap}, see, e.g., \cite[Lemma 1]{ramanathan1995incompleteness}, \cite[Section 3]{christensen1999density}, or \Cref{lem:hap_coherent} for a general result.

Applying \Cref{thm:main} to the setting described above yields the following new theorem on Gabor frames. It improves the result of Balan, Casazza, and Landau  on Gabor systems \cite[Theorem 1.2]{balan2011redundancy}, where it is assumed that the generating function belongs to the modulation space $M^1 (\mathbb{R}^d)$.

\begin{theorem} \label{thm:gabor}
 Let $g \in L^2 (\mathbb{R}^d)$ be nonzero. For every $\varepsilon > 0$, there exists $\Lambda \subseteq \mathbb{R}^{2d}$ satisfying $D^+(\Lambda) \leq 1+ \varepsilon$ and such that $\{ \pi (\lambda) g \}_{\lambda \in \Lambda}$ is a frame for $L^2 (\mathbb{R}^d)$.
\end{theorem}
\begin{proof}
Let $\phi_0(x)= 2^{d/4} e^{-\pi |x|^2}$, $x\in \R^d$. Without loss of generality, we may assume that $g \in L^2 (\mathbb{R}^d)$ is a unit vector. Then $\{ \pi(z) g \}_{z \in \mathbb{R}^{2d}}$ is a Parseval frame for $L^2 (\mathbb{R}^d)$, and hence so is $\{ V_{\phi_0} \pi(z) g \}_{z \in \mathbb{R}^{2d}}$ for $\mathcal{H} = V_{\phi_0} (L^2 (\mathbb{R}^d))$. Since the vectors $\{ V_{\phi_0} \pi(z) g \}_{z \in \mathbb{R}^{2d}}$ satisfy the weak localization \eqref{eq:wl} and the homogeneous approximation property \eqref{eq:hap}, they also satisfy the frame measure/density property \eqref{eq:frd} by \Cref{thm:framemeasure_rkhs}. Hence, an application of Theorem \ref{thm:main} yields a set $\Lambda \subseteq \mathbb{R}^{2d}$ satisfying $D^+ (\Lambda) \leq 1+\varepsilon$ and such that $\{ V_{\phi_0} \pi(\lambda) g \}_{\lambda \in \Lambda}$ is a frame for $\mathcal{H}$, which implies that also $\{\pi(\lambda) g \}_{\lambda \in \Lambda}$ is a frame for $L^2 (\mathbb{R}^d)$.
\end{proof}

\subsection{Coherent systems} \label{sec:coherent}
In this subsection, we consider a general class of frames arising from unitary representations for which density conditions have been studied in, e.g.,  \cite{caspers2023overcompleteness, enstad2025dynamical, enstad2025coherent, fuhr2007sampling, Papageorgiou2025counting, grochenig2008homogeneous, fuehr2017density}. This class contains the previously discussed exponential systems and Gabor systems as particular cases, but in contrast to these two examples the homogeneous approximation property \eqref{eq:hap} or the frame measure/density property \eqref{eq:frd} are in general not known to hold without a mild additional assumption on the generating vector, which we discuss in detail in this section.

Let $G$ be a second countable locally compact group with left Haar measure $\mu$. We assume that $G$ is noncompact, so that $\mu(G) = \infty$, but do assume that $G$ is compactly generated, i.e., there exists a symmetric relatively compact unit neighborhood $U \subseteq G$ such that $G = \bigcup_{n \in \mathbb{N}} U^n$ with $U^n = U \cdot \hdots \cdot U$ denoting the $n$-fold product of $U$. Moreover, we assume that $G$ has \emph{polynomial growth}, in the sense that there exist constants $C,D > 0$ such that
\begin{align} \label{eq:pg}
 \mu(U^n) \leq C n^D, \quad n \in \mathbb{N}.
\end{align}
 Any group of polynomial growth is unimodular.

Following \cite[Section 4.1]{breuillard2014geometry}, we say that a 
metric $d : G \times G \to [0, \infty)$ is \emph{periodic} if it satisfies the following properties:
\begin{enumerate}
 \item[(m1)] left-invariance, i.e., $d(x,y) = d(zx, zy)$ for all $x,y,z \in G$;
 \item[(m2)] metrically proper, i.e., the preimage of a bounded set of the map $y \mapsto d(e,y)$ is bounded;
 \item[(m3)] locally bounded, i.e., the image under $d$ of any compact subset of $G \times G$ is bounded;
 \item[(m4)] asymptotically geodesic, i.e., for every $\varepsilon > 0$, there exists $s>0$ such that, for any $x,y \in G$, there exists a sequence of points $x_1 = x, x_2, ..., x_n = y$ in $G$ such that
 \[
  \sum_{i = 1}^{n-1} d(x_i, x_{i+1}) \leq (1+\varepsilon) d(x,y)
 \]
 and $d(x_1, x_{i+1} ) \leq s$ for all $i = 1, ..., n - 1$.
\end{enumerate}
Examples of periodic metrics on $G$ include, for example, a word metric associated to a compact symmetric generating set, and a left-invariant Riemannian metric or a Carnot-Carath\'eodory metric whenever $G$ is a connected Lie group, cf. \cite[Section 4.3]{breuillard2014geometry}.

The significance of a periodic metric for our purposes is that for such metrics $d$ the triple $(G, d, \mu)$ satisfies all our standing assumptions on the metric measure spaces; see \cite[Corollary 1.6]{breuillard2014geometry} for the weak annular decay property \eqref{eq:wad} and the doubling at large scales \eqref{eq:doubling}.

For defining reproducing kernel Hilbert spaces, we consider a projective unitary representation $(\pi, \Hpi)$ of $G$ on a separable Hilbert space $\Hpi$, i.e., a strongly measurable map $\pi : G \to \mathcal{U}(\Hpi)$ such that $\pi(e) = I_{\Hpi}$ and
\[
 \pi(x)\pi(y) = \sigma(x,y) \pi(xy), \quad x,y \in G,
\]
for a function $\sigma : G \times G \to \mathbb{T}$. Given a vector $\eta \in \Hpi$, we define the associated map $V_{\eta} : \Hpi \to L^{\infty} (G)$ by
\[
 V_{\eta} f(x) = \langle f, \pi(x) \eta \rangle, \quad x \in G.
\]
For any $f, \eta \in \Hpi$, the function $|V_{\eta} f| : G \to [0, \infty)$ is continuous, cf. \cite[Lemma 7.1 and Theorem 7.5]{varadarajan1985geometry}.

A vector $\eta \in \Hpi$ is called \emph{admissible} if $V_{\eta} : \Hpi \to L^2 (G)$ is an isometry. For example, if $\pi$ is irreducible, i.e., the only closed $\pi$-invariant subspaces of $\Hpi$ are $\{0\}$ and $\Hpi$, then any vector $\eta \in \Hpi$ such that $V_{\eta} \eta \in L^2 (G)$ is an admissible vector if and only if $\| \eta \|^2 = d_{\pi}$, where $d_{\pi} > 0$ is the formal dimension of $\pi$, see, e.g., the orthogonality relations in \cite{aniello2006square}.

If $\eta \in \Hpi$ is an admissible vector, then we consider the (closed) subspace $\mathcal{H} := V_{\eta} (\Hpi)$ of $L^2 (G)$. A direct calculation entails that, for $f \in \Hpi$ and $x \in G$,
\[
 V_{\eta} \pi(x) f = L^{\sigma}_x V_g f,
\]
where $L^{\sigma}_x$ denotes the \emph{$\sigma$-twisted left translation operator} on $L^2 (G)$ given by \[ (L^{\sigma}_x F)(y) = \sigma(x, x^{-1} y) F(x^{-1} y) , \quad x, y \in G. \]
Therefore, if $F \in \mathcal{H}$, say $F = V_{\eta} f$, then
\[
 F(x) = \langle f, \pi (x) \eta \rangle = \langle V_{\eta} f, V_{\eta} \pi(x) \eta \rangle = \langle V_{\eta} f, L_x^{\sigma} V_{\eta} \eta \rangle,
\]
which shows that $\mathcal{H}$ is a reproducing kernel Hilbert space with reproducing kernel
\[
 k(x,y) = \overline{L_x^{\sigma} V_{\eta} \eta (y)}, \quad x, y \in G.
\]
In particular,
$
 k(x,x) = \| \eta \|^2,
$
so that $k$ satisfies the diagonal condition \eqref{eq:dc}. In addition, since $V_{\eta} g \in L^2 (G)$ for arbitrary $g \in \Hpi$, it is readily verified that any system of vectors $\{ g_x \}_{x \in G}$ of the form $g_x := V_{\eta} \pi (x) g$ satisfies the weak localization property \eqref{eq:wl}.

The only condition that is not generally satisfied in the present setting is the homogeneous approximation property \eqref{eq:hap}. In order to satisfy this condition, we consider the \emph{local maximal function} of $V_{\eta} f$ defined by
\begin{align} \label{eq:amalgam}
 M V_{\eta} f (x) := \sup_{y \in B_{r_0} (e)} |V_{\eta} f(x y)|,
 \end{align}
where $r_0 > 0$ is some fixed value such that $B_{r_0} (e)$ is a symmetric unit neighbourhood. Then we define the space
\[
\mathcal{B}^2_{\pi} := \big\{ \eta \in \Hpi : M V_{\eta} f \in L^2 (G), \; \forall f \in \Hpi \big\}.
\]
If $\pi$ is irreducible, then admissible vectors in $\mathcal{B}_{\pi}^2$ can always be shown to exist and be dense in $\Hpi$, cf. \cite[Remark 4.2]{fuhr2007sampling} or \cite[Remark 1]{grochenig2008homogeneous}.

For providing a sufficient condition for the homogeneous approximation property \eqref{eq:hap}, we state the following lemma, see, e.g., \cite[Lemma 1]{grochenig2008homogeneous} or \cite[Lemma 2.1]{Papageorgiou2025counting}.

\begin{lemma}[\cite{grochenig2008homogeneous}] \label{lem:sampling}
Let $r_0 > 0$ be such that $B_{r_0} (e)$ is a unit neighborhood. Let $F : G \to \mathbb{C}$ be continuous and let $r > r_0$. If $\Lambda \subseteq G$ is a relatively separated set, then
\[
 \sum_{\lambda \in \Lambda \setminus B_r (e)} |F(\lambda)|^2 \leq \frac{\sup_{y \in G} \# (\Lambda \cap B_{r_0}(y))}{\mu(B_{r_0}(e))} \int_{G\setminus B_{r-r_0}(e)} |MF(z)|^2 \; d\mu(z).
\]
\end{lemma}

Using \Cref{lem:sampling}, we prove the following lemma, see, e.g., \cite{grochenig2008homogeneous, enstad2025dynamical} for similar arguments.

\begin{lemma} \label{lem:hap_coherent}
 If $g \in \mathcal{B}_{\pi}^2$ is an admissible vector, then the system $\{g_x\}_{x \in G}$ of vectors $g_x = V_g \pi(x) g$ in $\mathcal{H} = V_g (\Hpi)$ satisfies the homogeneous approximation property \eqref{eq:hap}.
 
 In addition, if $G$ is abelian, then for any $h \in \Hpi$, the vectors $\{h_x\}_{x \in G} = \{V_g \pi(x) h\}_{x \in G}$ in $\mathcal{H} = V_g (\Hpi)$ satisfy the homogeneous approximation property \eqref{eq:hap}.
\end{lemma}
\begin{proof}
Let $\Lambda \subseteq G$ be a countable set such that $\{\gf\}_{\lambda \in \Lambda}$ is a Bessel sequence in $\mathcal{H}$. Then $\Lambda \subseteq G$ is relatively separated by Lemma \ref{lem:relatively_dense}. Fix $\varepsilon > 0$. Since $M V_g g \in L^2 (G)$ by assumption, we can choose $r > 0$ such that
\[
 \frac{\sup_{y \in G} \# (\Lambda \cap B_{r_0}(y))}{\mu(B_{r_0}(e))} \int_{G\setminus B_{r-r_0}(e)} |MV_gg(z)|^2 \; d\mu(z) < \varepsilon^2.
\]
For $x \in G$, a direct calculation gives
\begin{align*}
 \sum_{\lambda \in \Lambda \setminus B_r(x)} |g_{\lambda} (x)|^2
 &= \sum_{\lambda \in \Lambda \setminus B_r (x)} |V_g  g(\lambda^{-1} x)|^2 \\
 &= \sum_{\lambda \in \Lambda \setminus B_r (x)} |V_g  g(x^{-1} \lambda)|^2 \\
 &= \sum_{\lambda \in (x^{-1} \Lambda) \setminus B_r (e)} |V_g  g( \lambda)|^2,
\end{align*}
where the second equality used that $|V_g g(y)| = |V_g g(y^{-1})|$ for all $y \in G$. Since $x^{-1} \Lambda $ is also relatively separated and
$
 \sup_{y \in G} \#(x^{-1} \Lambda \cap B_r(y)) = \sup_{y \in G} \#( \Lambda \cap B_r(y))
$, an application of Lemma \ref{lem:sampling} with $F=|V_gg|$ yields
\[
 \sum_{\lambda \in \Lambda \setminus B_r(x)} |g_{\lambda} (x)|^2 \leq \frac{\sup_{y \in G} \# (\Lambda \cap B_{r_0}(y))}{\mu(B_{r_0}(e))} \int_{G\setminus B_{r-r_0}(e)} |(MV_gg)(z)|^2 \; d\mu(z) < \varepsilon^2,
\]
which shows that $\{g_x\}_{x \in G}$ satisfies the homogeneous approximation property \eqref{eq:hap}.

For the additional part, note that if $G$ is abelian and $h \in \Hpi$, then
\[
M V_h g (x) = \sup_{y \in B_{r_0} (e)} |\langle g, \pi(x y) h \rangle| = \sup_{y \in B_{r_0} (e)} |\langle h, \pi(y^{-1} x^{-1} ) g \rangle | = M V_g h (x^{-1}),
\]
where the last step used that $B_{r_0} (e)$ is symmetric. Since $M V_{g} h \in L^2 (G)$ by assumption, the above calculation and the fact that $G$ is abelian show that
\begin{align*}
\int_G |M V_h g (x) |^2 \; d\mu(x) &= \int_G |M V_g h (x^{-1}) |^2 \; d\mu(x) = \int_G |MV_g h(x)|^2 \; d\mu(x) 
< \infty. 
\end{align*}
 Thus, $MV_{h} g \in L^2 (G)$ whenever $g \in \mathcal{B}_{\pi}^2$ and $h \in \Hpi$. Therefore, a similar calculation as above gives that
\begin{align*}
\sum_{\lambda \in \Lambda \setminus B_r (x)} | h_{\lambda} (x)|^2= \sum_{\lambda \in \Lambda \setminus B_r (x)} | V_h g (x^{-1} \lambda )|^2 
&= \sum_{\lambda \in (x^{-1}\Lambda ) \setminus B_r (e)} | V_h g ( \lambda )|^2,
\end{align*}
so that an application of Lemma \ref{lem:sampling} shows
\[
 \sum_{\lambda \in \Lambda \setminus B_r (x)} | h_{\lambda} (x)|^2 \leq
\frac{\sup_{y \in G} \# (\Lambda \cap B_{r_0} (y))}{\mu(B_{r_0}(e))} \int_{G \setminus B_{r-r_0}(e)} |MV_h g (z)|^2 \; d\mu(z),
\]
which yields that $\{h_x\}_{x \in G}$ satisfies the homogeneous approximation property \eqref{eq:hap}.
\end{proof}

Combining the above observations, we obtain the following theorem on the existence of coherent frames near the critical density, which shows that the necessary density conditions for such frames obtained in, e.g., \cite{caspers2023overcompleteness, enstad2025dynamical, enstad2025coherent, Papageorgiou2025counting, fuehr2017density}, are optimal. In particular, it solves the open question \cite[Question 11]{velthoven2024density} posed by the second author.

\begin{theorem}
Let $(\pi, \Hpi)$ be a projective unitary representation of a compactly generated noncompact group $G$ of polynomial growth. Suppose $g \in \mathcal{B}_{\pi}^2$ is an admissible vector. For every $\varepsilon >0$, there exists $\Lambda \subseteq G$ satisfying
\[
  D^+(\Lambda) \leq (1+\varepsilon) \| g \|^2
\]
and such that
$\{\pi(\lambda) g\}_{\lambda \in \Lambda}$ is a frame for $\Hpi$.
\end{theorem}
\begin{proof}
Since $g \in \Hpi$ is admissible, the system $\{\pi(x) g\}_{x \in G}$ forms a Parseval frame for $\Hpi$, and hence $\{g_x \}_{x \in G} = \{V_{g} \pi(x) g\}_{x \in G}$ is a Parseval frame for $\mathcal{H} = V_g (\Hpi)$. Moreover, since $g \in \mathcal{B}^2_{\pi}$, the system $\{g_x \}_{x \in G}$ satisfies the homogeneous approximation property \eqref{eq:hap} by Lemma \ref{lem:hap_coherent}. Thus an application of \Cref{thm:framemeasure_rkhs} yields that $\{g_x \}_{x \in G}$ satisfies the frame measure/density property \eqref{eq:frd}. Theorem \ref{thm:main} is therefore applicable, and yields an index set $\Lambda \subseteq G$ such that $\{\gf \}_{\lambda \in \Lambda}$ is a frame for $\mathcal{H}$ with density
\[
 D^+ (\Lambda) \leq  (1+\varepsilon) \| g \|^2.
\]
Since $V_g : \Hpi \to \mathcal{H}$ is unitary, it follows that also $\{ \pi(\lambda) g\}_{\lambda \in \Lambda}$ is a frame for $\Hpi$.
\end{proof}

\begin{corollary}
 If $\pi$ is an irreducible projective unitary representation of $G$ and $g \in \mathcal{B}_{\pi}^2$ is nonzero, then for any $\varepsilon > 0$ there exists $\Lambda \subseteq G$ such that $D^+(\Lambda) \leq (1+\varepsilon) d_{\pi}$ and such that $\{\pi(\lambda) g \}_{\lambda \in \Lambda}$ is a frame for $\Hpi$.
\end{corollary}

\subsection{Spectral subspaces of elliptic differential operators} \label{sec:elliptic}
In this subsection, we apply our main theorem to spectral subspaces of elliptic differential operators as recently studied in \cite{grochenig2024necessary}; see also \cite{grochenig2017what} for the particular case of spaces of variable bandwidth on the real line.

For $s \geq 0$, the Sobolev space $W^s_2 (\mathbb{R}^d)$ is as usual defined to be subspace of $L^2 (\mathbb{R}^d)$ of elements $f \in L^2 (\mathbb{R}^d)$ satisfying
\[
 \int_{\mathbb{R}^d} |\widehat{f}(\xi)|^2 (1+|\xi|^2)^s \; d\xi < \infty.
\]
Let $a \in C_b^{\infty} (\mathbb{R}^d, \mathbb{C}^{d\times d})$ be positive definite, that is, $a_{j,k} = \overline{a_{k,j}} \in C^{\infty}_b (\mathbb{R}^d)$ and there exists $\theta > 0$ such that
\[
 a(x) \xi \cdot \xi \geq \theta |\xi|^2 \quad x, \xi \in \mathbb{R}^d.
\]
Then the differential operator $H_a$ of the form
\[
 H_a f = - \sum_{j,k=1}^d \partial_j a_{j,k} \partial_k f, \quad f \in W^2_2 (\mathbb{R}^d)
\]
is a postive, uniformly elliptic self-adjoint operator on $L^2 (\mathbb{R}^d)$ with domain $\mathcal{D}(H_a) = W^2_2 (\mathbb{R}^d)$. For $\Omega > 0$, denote by $\mathds{1}_{[0, \Omega]} (H_a)$ the spectral projection of $H_a$ corresponding to the spectrum $[0, \Omega]$. Following \cite{grochenig2024necessary}, the \emph{Paley-Wiener space} associated to $H_a$ is the spectral subspace
\[
 \PW_{\Omega} (H_a) := \mathds{1}_{[0, \Omega]} (H_a) L^2 (\mathbb{R}^d).
\]
For the particular case that \[
     H_a f = - \frac{d}{dx} \bigg(a \frac{d}{dx} f \bigg),
                             \]
the Paley-Wiener spaces associated to $H_a$ corresponds to spaces of variable bandwidth studied in \cite{grochenig2017what}, cf. \cite[Example 6D]{grochenig2024necessary}.

By \cite[Lemma 2.1, Proposition 2.2]{grochenig2024necessary}, the Paley-Wiener space $\PW_{\Omega} (H_a)$ is a reproducing kernel Hilbert space that is continuously embedded into all Sobolev spaces $W^s_2(\mathbb{R}^d)$, $s \geq 0$, and its reproducing kernel $k$ satisfies the diagonal condition \eqref{eq:dc}. Moreover, if the symbol $a \in C^{\infty}_b (\mathbb{R}^d, \mathbb{C}^{d \times d})$ is \emph{slowly oscillating}, in the sense that $\lim_{|x|\to \infty} |\partial_k a(x) |=0$ for all $k = 1, ..., d$, then the reproducing kernels $\{k_x\}_{x \in \mathbb{R}^d}$ satisfy the weak localization property \eqref{eq:wl} and the homogeneous approximation property \eqref{eq:hap}, cf. \cite[Section 2D2]{grochenig2024necessary} and \cite[Theorem 4.4, Theorem 5.2]{grochenig2024necessary}. Lastly, $\mathbb{R}^d$ equipped with Lebesgue measure and Euclidean metric clearly satisfy our standing assumptions on the metric measure space.

By the previous paragraph, the results in the present paper are applicable to the reproducing kernel Hilbert spaces $\PW_{\Omega} (H_a)$, which directly yields the following result. This shows the existence of optimal sampling sets in such spaces and resolves an open question mentioned on \cite[p. 590]{grochenig2024necessary}. The result appears to be new even for spaces of variable bandwidth as considered in \cite{grochenig2017what}.

\begin{theorem}
 Let $a \in C_b^{\infty} (\mathbb{R}^d, \mathbb{C}^{d\times d})$ be positive definite and slowly oscillating and $\Omega > 0$. For every $\varepsilon > 0$, there exists $\Lambda \subseteq \mathbb{R}^d$ satisfying
 \[
  D^+_0(\Lambda) \leq 1+\varepsilon
 \]
 and such that the reproducing kernels $\{k_{\lambda} \}_{\lambda \in \Lambda}$ form a frame for $\PW_{\Omega} (H_a)$.
\end{theorem}
\begin{proof}
Since the reproducing kernels $\{k_x\}_{x \in \mathbb{R}^d}$ satisfy the weak localization property \eqref{eq:wl} and the homogeneous approximation property \eqref{eq:hap}, they also satisfy the frame measure/density property \eqref{eq:frd} by \Cref{thm:framemeasure_rkhs}. In addition, the diagonal condition \eqref{eq:dc} yields that $C_1 \leq \|k_x\|^2 \leq C_2$ for all $x \in X$. Hence, applying \Cref{thm:main} to the Parseval frame $\{k_x \}_{x \in x}$ yields the result.
\end{proof}

\subsection{Weighted spaces of entire functions} \label{sec:holomorphic}
In this subsection, we rederive the existence of optimal sampling sets in certain weighted spaces of entire functions, cf. \cite{grochenig2019strict}.

We consider $X = \mathbb{C}^d$ equipped with Lebesgue measure $\mu$ and Euclidean distance. For a twice continuously differentiable $\phi : \mathbb{C}^d \to \mathbb{R}$, we assume that there exist two constants $m , M > 0$ such that
\begin{align} \label{eq:pluri}
 m \mathbf{I}_d \leq \big( \partial_j \overline{\partial_k} \phi(z) \big)_{j,k \in \{1, ..., d\}} \leq M \mathbf{I}_d , \quad z \in \mathbb{C}^d \setminus \{0\},
\end{align}
in the sense of positive definite matrices. The associated \emph{weighted Fock space} $\mathcal{F}^2_{\phi} (\mathbb{C}^d)$ is the space of entire functions $f$ on $\mathbb{C}^d$ satisfying
\[
\| f \|_{\mathcal{F}^2_{\phi}}^2 = \int_{\mathbb{C}^d} |f(z)|^2 e^{-2 \phi (z)} \; dz < \infty,
\]
where $dz$ denotes the Lebesgue measure on $\mathbb{C}^d$.
Then $\mathcal{F}^2_{\phi} (\mathbb{C}^d)$ is a reproducing kernel Hilbert space whose reproducing kernel  $k_{\phi}$ satisfies
\begin{align} \label{eq:weighted_diagonal}
 0 < C_1 \leq k_{\phi} (z,z) e^{-2 \phi (z)} \leq C_2 < \infty \quad z \in \mathbb{C}^d,
\end{align}
and
\begin{align} \label{eq:weighted_offdiagonal}
 |k_{\phi} (z,w) |e^{- \phi(z) - \phi(w)} \leq c_1 e^{-c_2 |z-w|}, \quad z, w \in \mathbb{C}^d,
\end{align}
where $c_1, c_2, C_1, C_2 > 0$ are constants that depend only on the constants in \eqref{eq:pluri}.
See \cite{lindholm2001sampling, schuster2012toeplitz, delin1998pointwise} and \cite[Section 2.2]{grochenig2019strict}.

The following theorem recovers the sampling part in \cite[Theorem 1.1]{grochenig2019strict}.

\begin{theorem}
Let $\phi \in C^2 (\mathbb{C}^d, \mathbb{R})$ be a function satisfying \eqref{eq:pluri}.

For every $\varepsilon > 0$, there exists $\Lambda \subseteq \mathbb{C}^d$ satisfying
\[
 D^+_{\phi} (\Lambda) := \liminf_{r \to \infty} \inf_{z \in \mathbb{C}^d} \frac{\# (\Lambda \cap B_r (z))}{\int_{B_r (z)} k_{\phi} (w,w) e^{-2 \phi (w)} \; d\mu (w)} \leq 1+\varepsilon
\]
and such that there exist constants $0 < A \leq B < \infty$ satisfying
\[
 A \| f \|^2_{\mathcal{F}^2_{\phi}} \leq \sum_{\lambda \in \Lambda} |f(\lambda) |^2 e^{-2 \phi(\lambda)} \leq B \| f \|^2_{\mathcal{F}^2_{\phi}}
\]
for all $f \in \mathcal{F}^2_{\phi} (\mathbb{C}^d)$.
\end{theorem}
\begin{proof}
We define the space $\mathcal{H} := \big\{ e^{- \phi} f : f \in \mathcal{F}^2_{\phi} (\mathbb{C}^d) \big\} \subseteq L^2 (\mathbb{C}^d, \mu)$, which is a reproducing kernel Hilbert space with kernel $\widetilde{k}_{\phi} (z,w) := k_{\phi} (z,w) e^{-\phi(z) - \phi(w)}$ for $z, w \in \mathbb{C}^d$. Hence, the estimates \eqref{eq:weighted_diagonal} imply that $\tilde{k}_{\phi}$ satisfies the diagonal condition \eqref{eq:dc}, whereas the weak localization property \eqref{eq:wl} and the homogeneous approximation property \eqref{eq:hap} follow from the off-diagonal decay estimate \eqref{eq:weighted_offdiagonal}; see also (the proof of) \cite[Theorem 2.1]{grochenig2019strict}. Thus, Theorem \ref{thm:framemeasure_rkhs} and Theorem \ref{thm:main} imply the existence of a set $\Lambda \subseteq \mathbb{C}^d$ satisfying $D_{\phi}^+ (\Lambda) = D^+_0(\Lambda) \leq 1+ \varepsilon$ and such that the reproducing kernels $\{\widetilde{k}_{\phi} (\cdot, \lambda) \}_{\lambda \in \Lambda}$ forms a frame for $\mathcal{H} \subseteq L^2 (\mathbb{C}^d)$.
\end{proof}

\appendix

\section{Proof of regularization lemma}
In this section we give the proof of Lemma \ref{regularization}. We start by giving the following convenient definition.

\begin{definition}
 Let $(X, d, \mu)$ be a metric measure space. A measure $\nu$ on $X$ is said to be \emph{density equivalent} with $\mu$ if it satisfies the following conditions:
 \begin{enumerate}[(d1)]
  \item The measure $\nu$ is absolutely continuous with respect to $\mu$ and there exists $C > 0$ such that
  \[
   \frac{1}{C} \leq \frac{d\nu (x)}{d\mu(x)} \leq C \quad \text{for $\mu$-a.e.} \quad x \in X;
  \]
\item
\[
\liminf_{r \to \infty} \inf_{x \in X} \frac{\nu (B_r (x))}{\mu(B_r (x))} = \limsup_{r \to \infty} \sup_{x \in X} \frac{\nu (B_r (x))}{\mu(B_r (x))} = 1.
\]
 \end{enumerate}
\end{definition}

Note that if $\nu$ is density equivalent with $\mu$, then also $(X, d, \nu)$ satisfies the standing assumptions in \Cref{sec:repro_measure} by (d1). Moreover, for any $\Lambda \subseteq X$ we have $D^+_{\mu}(\Lambda)=D^+_\nu(\Lambda)$ and $D^-_{\mu}(\Lambda)=D^-_\nu(\Lambda)$ by (d2).

\begin{lemma}\label{step}
  Let $\Lambda \subseteq X$ be such that $0<D^-_{\mu}(\Lambda) \le D^+_{\mu}(\Lambda)<\infty$ and let $\ve>0$. Let $\{X_y\}_{y\in Y}$ be a partition of $X$ satisfying \eqref{xy} and \eqref{eq:stephyp2} for some constants $R_0>0$ and $C_0>1$.
Suppose that $\{X^{(1)}_y\}_{y\in Y_1}$ is a coarser partition of $X$ than $\{X_y \}_{y \in Y}$, i.e., for all $y \in Y$ there exists $y_1 \in Y$ such that $X_y \subseteq X_{y_1}^{(1)}$, and such that
\begin{equation}\label{xy1}
 B_{R_1}(y)  \subseteq X^{(1)}_{y} \subseteq B_{3R_1}(y) \qquad \text{for all }y\in Y_1.
 \end{equation}
for some $R_1 > 0$.

Then, for any sufficiently large $R_2>0$ and any family $\{B_{R_2}(y)\}_{y\in Y_2}$ of disjoint balls in $X$, there exists a measure $\nu$ on $X$, which is density equivalent with $\mu$, such that the set
 \[
 \widetilde{Y}^{(1)} = \big\{y \in Y: \exists y_1\in Y_1, \; \exists y_2 \in Y_2, \quad X_y \subseteq X_{y_1}^{(1)} \subseteq B_{R_2}(y_2) \big\}
 \]
 satisfies:
 \begin{enumerate}[(i)]
\item
 for all $y\in \widetilde Y^{(1)}$, we have
 \[
 D^-_{\mu}(\Lambda) - \ve \le \frac{\#(\Lambda \cap X_y)}{\nu(X_y)} \le D^{+}_{\mu}(\Lambda) + \ve.
 \]
 \item  $\nu(X_y)=\mu(X_y)$ for all $y \in Y \setminus \widetilde{Y}^{(1)}$,
 \item $\nu(\widetilde{X}_{y_2})=\mu(\widetilde{X}_{y_2})$ for all $y_2 \in Y_2$, where $\widetilde{X}_{y_2}= \bigcup_{y \in Y_1 :\  X^{(1)}_y \subseteq B_{R_2} (y_2)} X^{(1)}_y$.
 \end{enumerate}
\end{lemma}

\begin{proof}
We split the proof into three steps.
\\~\\
\textbf{Step 1.} In this step, we will define the measure $\nu$.
For this, let $\eta > 0$ be sufficiently small to be determined later.
Then,
by the weak annular decay property (\Cref{lem:wad}), there exists $R_2>0$ satisfying
\begin{align} \label{eq:wadeta1}
 \frac{\mu(B_{R_2} (x) \setminus B_{R_2 - 6R_1} (x))}{\mu(B_{R_2} (x))} < \eta \qquad \text{for all } x \in X.
\end{align}
By increasing $R_2 > 0$ if necessary, we can further assume that for all $R\ge R_2 -6R_1$ we have
\begin{equation}\label{step3}
 D^-_{\mu} (\Lambda) - \frac{\varepsilon}{2} \leq \frac{\#(\Lambda \cap B_{R} (x))}{\mu(B_{R} (x))} \leq D^+_{\mu} (\Lambda) + \frac{\varepsilon}{2} \qquad \text{for all }  x \in X.
\end{equation}
Fix some family $\{B_{R_2}(y_2)\}_{y_2\in Y_2}$ of disjoint balls in $X$. For $y_2 \in Y_2$, define
\[
 \widetilde{X}_{y_2} := \bigcup_{y \in Y_1 :\  X^{(1)}_y \subseteq B_{R_2} (y_2)} X^{(1)}_y.
\]
By \eqref{xy1} the diameter of a set $X^{(1)}_y$ is at most $6R_1$ for any $y\in Y_1$. Hence, 
\begin{equation}\label{step4}
B_{R_2 - 6R_1} (y_2) \subseteq \widetilde{X}_{y_2} \subseteq B_{R_2} (y_2) \qquad\text{for all } y_2 \in Y_2.
\end{equation}
Note that \eqref{eq:wadeta1} can be rewritten as
\begin{equation}\label{step5}
 \frac{\mu (B_{R_2-6R_1} (x))}{\mu(B_{R_2} (x))} > 1 - \eta \qquad \text{for all }  x \in X.
\end{equation}
Hence, using \eqref{step3}, \eqref{step4}, and \eqref{step5} yields for all $y_2\in Y_2$,
\begin{align*}
(1-\eta) (D^-_{\mu} (\Lambda) - \varepsilon/2) &\leq  (1-\eta) \frac{\#(B_{R_2-6R_1}(y_2) \cap \Lambda)}{\mu(B_{R_2-6R_1 } (y_2))}
\leq  \frac{\#(B_{R_2-6R_1}(y_2) \cap \Lambda)}{\mu(B_{R_2 } (y_2))}
\\
&\leq \frac{\#(\widetilde{X}_{y_2} \cap \Lambda)}{\mu(\widetilde{X}_{y_2})}   \leq \frac{\#(B_{R_2}(y_2) \cap \Lambda)}{\mu(B_{R_2 - 6R_1} (y_2))}
\leq \frac{\#(B_{R_2}(y_2) \cap \Lambda)}{(1-\eta)\mu(B_{R_2} (y_2))} \\
&\leq  \frac{D^+_{\mu} (\Lambda) + \varepsilon/2}{1-\eta}.
\end{align*}
Therefore,  for a choice of sufficiently small $\eta>0$, we have
\begin{align} \label{eq:lemstepb}
 D^-_{\mu} (\Lambda) - \varepsilon \leq \frac{\#(\widetilde{X}_{y_2} \cap \Lambda)}{\mu(\widetilde{X}_{y_2})}  \leq D^+_{\mu} (\Lambda) + \varepsilon \quad \text{for all} \quad y_2 \in Y_2.
\end{align}
For $y \in Y$, define
\begin{align*}
 a_y =
 \begin{cases}
  \frac{\#(\Lambda \cap X_y)}{\mu(X_y)} \frac{\mu (\widetilde{X}_{y_2})}{\# (\widetilde{X}_{y_2} \cap \Lambda)},  \quad & X_y \subseteq \widetilde X_{y_2} \;\; \text{for some} \;\; y_2 \in Y_2, \\
  1, & \text{otherwise};
 \end{cases}
\end{align*}
and set
$
 d \nu := \big( \sum_{y \in Y} a_y \mathds{1}_{X_y} \big) \cdot d\mu.
$
\\~\\
\textbf{Step 2.} In this step, we show that $\nu$ satisfies (i)--(iii) of the statement of the lemma. For this, note that if $y \in \widetilde{Y}^{(1)}$, so that $X_y \subseteq \widetilde X_{y_2}$ for some $y_2 \in Y_2$, then a direct calculation shows that
\[
 \frac{\#(X_y \cap \Lambda)}{\nu (X_y)} = \frac{\#(\widetilde{X}_{y_2} \cap \Lambda)}{\mu (\widetilde{X}_{y_2})}.
\]
Thus (i) follows from \eqref{eq:lemstepb}. On the other hand, if $y \in Y \setminus \widetilde{Y}^{(1)}$, then $a_y=1$ and thus
$\nu(X_y)=\mu(X_y)$, which shows (ii). In particular, for every measurable set $E \subseteq X \setminus \bigcup_{y_2 \in Y_2} \widetilde{X}_{y_2}$, we have $\nu(E) = \mu(E)$. Finally, for all $y_2 \in Y_2$,
\[
 \nu (\widetilde{X}_{y_2}) = \sum_{y \in Y_1 : \ X_y \subseteq \widetilde X_{y_2}} \nu (X_y) = \sum_{y \in Y_1 : \ X_y \subseteq \widetilde X_{y_2} } \frac{\#(X_{y} \cap \Lambda) \mu(\widetilde{X}_{y_2})}{\#(\widetilde{X}_{y_2} \cap \Lambda)} =   \mu (\widetilde{X}_{y_2}),
\]
which shows (iii).
\\~\\
\textbf{Step 3.} In this step, we show that $\nu$ is density equivalent to $\mu$. Without loss of generality we can assume that $\ve< D^-_{\mu} (\Lambda)/2$.
The definition of $\nu$ immediately yields that $\nu$ is absolutely continuous with respect to $\mu$.  Using \eqref{eq:stephyp2} and \eqref{eq:lemstepb}, it follows that, for $y \in Y$,
\[
 a_{y} \leq C_0 D^+_{\mu} (\Lambda) \frac{1}{D^-_{\mu} (\Lambda) - \varepsilon} \leq 2 C_0 \frac{D^+_{\mu} (\Lambda)}{D^-_{\mu} (\Lambda)} =: C'.
\]
Similarly, it follows that $a_{y} \geq 1/C'$ for all $y \in Y$. Therefore,
\begin{equation}\label{radon}
 \frac{1}{C'} \leq \frac{d\nu}{d\mu} (x) \leq C' \quad \text{for a.e.} \quad x \in X.
\end{equation}

Finally, we show that
\begin{align}\label{eq:preservation}
 \liminf_{r \to \infty} \inf_{x \in X} \frac{\nu (B_r(x))}{\mu(B_r(x))} = \limsup_{r \to \infty} \sup_{x \in X} \frac{\nu (B_r(x))}{\mu(B_r(x))} =1.
\end{align}
For $x \in X$ and $r>0$ sufficiently large, we have that
\begin{align*}
 \nu (B_r (x)) &= \nu \bigg(B_r(x) \setminus \bigcup_{y_2 \in Y_2 :\ \widetilde{X}_{y_2} \subseteq B_r(x)} \widetilde{X}_{y_2} \bigg) + \nu \bigg( \bigcup_{y_2 \in Y_2 :\ \widetilde{X}_{y_2} \subseteq B_r(x)} \widetilde{X}_{y_2} \bigg) \\
 &= \nu \bigg( \bigg(B_r(x) \setminus \bigcup_{y_2 \in Y_2 :\ \widetilde{X}_{y_2} \subseteq B_r(x)} \widetilde{X}_{y_2} \bigg) \cap \bigcup_{y_2 \in Y_2} \widetilde{X}_{y_2} \bigg) \\
 &\quad \quad  +  \nu \bigg( \bigg(B_r(x) \setminus \bigcup_{y_2 \in Y_2 :\ \widetilde{X}_{y_2} \subseteq B_r(x)} \widetilde{X}_{y_2} \bigg) \cap \bigg( X \setminus \bigcup_{y_2 \in Y_2} \widetilde{X}_{y_2} \bigg) \bigg)   \\
 &\quad \quad + \nu \bigg( \bigcup_{y_2 \in Y_2 :\ \widetilde{X}_{y_2} \subseteq B_r(x)} \widetilde{X}_{y_2} \bigg).
 \end{align*}
 By \eqref{step4}, we have
 \[
 \bigg( B_r(x) \setminus \bigcup_{y_2 \in Y_2 :\  \widetilde{X}_{y_2} \subseteq B_r(x)} \widetilde{X}_{y_2} \bigg) \cap \bigcup_{y_2 \in Y_2} \widetilde{X}_{y_2} \subseteq B_r (x) \setminus B_{r-6R_1} (x).
\]
Hence, using the fact that $\nu(E) = \mu(E)$ for every measurable set $E \subseteq X \setminus \bigcup_{y_2 \in Y_2} \widetilde{X}_{y_2}$ and conclusion (iii), we have
 \begin{align*} 
 \nu(B_r(x)) &\leq  \mu(B_r (x)) + \nu (B_r(x) \setminus B_{r-6R_1}(x)).
\end{align*}
By \eqref{radon} the measure $\nu$ satisfies the weak annular property. Hence,  
\[
\liminf_{r \to \infty} \inf_{x \in X} \frac{\mu (B_r(x))}{\nu(B_r(x))} \ge 1.
\] 
In a similar manner, one can show that
\[
 \mu (B_r (x)) \leq \nu(B_r (x)) + \mu (B_r(x) \setminus B_{r-6R_1}(x)) .
\]
Hence, 
\[
\limsup_{r \to \infty} \sup_{x \in X} \frac{\mu (B_r(x))}{\nu(B_r(x))} = \bigg( \liminf_{r \to \infty} \inf_{x \in X} \frac{\nu (B_r(x))}{\mu(B_r(x))}\bigg)^{-1} \le 1.
\]
Thus, the claim \eqref{eq:preservation} follows.
\end{proof}

\begin{proof}[Proof of Lemma \ref{regularization}]
We split the proof into several steps.
\\~\\
\textbf{Step 1.} This step will construct certain partitions $\{X_y^{(n)} \}_{y \in Y_n}$, $n \in \mathbb{N}_0$, of $X$.

We start by constructing recursively centers $0 < R_0 < R_1 < ... < R_n$ and associated sets $Y_n \subseteq X$ for $n \in \mathbb{N}_0$. 
For the base case we let $X_y^{(0)}=X_y$ for $y\in Y_0:=Y$. We will choose $R_1>0$ such that doubling property \eqref{eq:doubling} holds for $r\ge R_1$.
For the recursive step we assume we are given radii $0 < R_0 < ... < R_{n-1}$ for $n \ge 1$. Then, by the weak annular decay property there exists $R_{n} > R_{n-1}$ such that for all $R>R_n$ we have
\begin{equation}\label{ann}
 \frac{\mu(B_{R} (x) \setminus B_{R- 9 R_{n-1}} (x) )}{\mu(B_{R} (x))} <\eta_{n} \qquad \text{for all } x \in X,
\end{equation}
where $0<\eta_{n}<1/2$ is sufficiently small constant to be chosen later in order to apply Lemma \ref{step} in Step 3.
For such $R_{n}$, we choose $Y_{n} \subseteq X$ such that $\{B_{R_{n}+6R_{n-1}} (y) \}_{y \in Y_{n}}$ is a maximal family of disjoint balls.

Next, we construct the associated partition $\{X_y^{(n)} \}_{y \in Y_n}$ of $X$ for $n \in \mathbb{N}_0$, which satisfy
\begin{equation}\label{xy2}
 B_{R_n}(y)  \subseteq X^{(n)}_{y} \subseteq B_{3R_n}(y) \qquad \text{for all }y\in Y_n.
 \end{equation}
  Note that $\{X_y^{(0)}\}_{y \in Y_0}$ satisfies \eqref{xy2} by assumption. For the recursive step, we assume that we have partitions $\{X_y^{(k)} \}_{y \in Y_k}$ for $0 \leq k \leq n-1$ for some $n \geq 1$. For $y \in Y_{n}$, we let
\begin{equation}\label{recur}
 \widetilde{X}^{(n)}_y := \bigcup_{z \in Y_n \; : \; X_z^{(n-1)} \subseteq B_{R_{n}+6R_{n-1}} (y)} X_z^{(n-1)}.
\end{equation}
On one hand, since $\diam (X_z^{(n-1)}) \leq 6 R_{n-1}$, 
we have that $B_{R_{n} } (y) \subseteq \widetilde{X}_y^{(n)}$. On the other hand, by construction,
$
 \widetilde{X}_y^{(n)} \subseteq B_{R_{n}+6R_{n-1}} (y)
$ for $y \in Y_{n}$.
Therefore, since $X = \bigcup_{y \in Y_{n}} B_{2 R_{n}+12R_{n-1}} (y)$, we can enlarge the sets $\widetilde{X}^{(n)}_y$ into sets $X^{(n)}_y$ such that
$\{X^{(n)}_y \}_{y \in Y_{n}}$ is a partition of $X$ satisfying
\begin{align} \label{eq:partition_nested}
& B_{R_{n}} (y) \subseteq X_y^{(n)} \subseteq B_{2R_{n} + 12R_{n-1}} (y)
 \qquad\text{for all } y \in Y_{n},
 \\
 \label{nestn}
&  \forall y \in Y_{n-1} \quad  \exists y' \in Y_{n} \quad X^{(n-1)}_y \subseteq X^{(n)}_{y'}.
\end{align}
By choosing $R_n>12R_{n-1}$, it follows that \eqref{xy2} holds.
In addition, the doubling condition of $\mu$ on large scales and \eqref{ann} yields
\begin{align} \label{quotient_doubling}
 \frac{\mu(\widetilde{X}_y^{(n)} )}{\mu(X_y^{(n)} )} \geq  \frac{ \mu(B_{R_{n}}(y))}{\mu(B_{2 R_{n}+12R_{n-1}}(y))} \geq
 \frac{1}{C_d} \frac{ \mu(B_{R_{n} }(y))}{ \mu(B_{ R_{n}+6R_{n-1}}(y))} \ge \frac{1-\eta_{n}}{C_d} \ge  \frac{1}{2 C_d}.
\end{align}
for all $y \in Y_n$ and $n \in \mathbb{N}$.
\\~\\
\textbf{Step 2.} For $n \in \mathbb{N}$, consider the sets
\[
 E_n := \bigcup_{k = 1}^n \bigcup_{y \in Y_k} \widetilde{X}_y^{(k)},
\]
where the sets $\widetilde{X}_y^{(k)}$ are as in Step 1. In this step, we will show by induction that, for any $n \in \mathbb{N}$,
\begin{align} \label{eq:claimEN_induction}
 \frac{\mu(X_y^{(n)} \cap E_n )}{\mu(X_y^{(n)})} \geq 1 - \bigg(1- \frac{1}{2 C_d } \bigg)^n \quad \text{for all} \quad y \in Y_n.
\end{align}
For $n = 1$, the claim follows directly from \eqref{quotient_doubling}. Assume next that \eqref{eq:claimEN_induction} holds for some $n \in \mathbb{N}$.  For $y \in Y_{n+1}$,
set $Z_y^{(n)} := \{ z \in Y_n : X_z^{(n)} \subseteq X_y^{(n+1)} \setminus \widetilde{X}_y^{(n+1)}\}$. Then
\begin{align*}
 \mu(X_y^{(n+1)} \cap E_{n+1}) &\geq \mu(\widetilde{X}_y^{(n+1)}) + \mu ((X_y^{(n+1)} \setminus \widetilde{X}_y^{(n+1)} ) \cap E_{n+1}) \\
 &= \mu(\widetilde{X}_y^{(n+1)}) + \mu\big( \bigcup_{z \in Z_y^{(n)}} X_z^{(n)} \cap E_n \big) \\
 &= \mu(\widetilde{X}_y^{(n+1)}) + \sum_{ z \in Z_y^{(n)}} \mu(X_z^{(n)} \cap E_n ) \\
 &\geq  \mu(\widetilde{X}_y^{(n+1)}) + \sum_{ z \in Z_y^{(n)}} \mu(X_z^{(n)}) \bigg( 1 - \bigg( 1 - \frac{1}{2C_d} \bigg)^n \bigg) \\
 &= \mu(\widetilde{X}_y^{(n+1)}) + \bigg( \mu(X_y^{(n+1)}) - \mu (\widetilde{X}_y^{(n+1)}) \bigg) \bigg( 1 - \bigg( 1 - \frac{1}{2C_d} \bigg)^n \bigg) \\
 &\geq \mu(X_y^{(n+1)}) \bigg( \frac{1}{2C_d} + \bigg(1 - \frac{1}{2C_d} \bigg) \bigg( 1 - \bigg( 1 - \frac{1}{2C_d} \bigg)^n \bigg) \bigg) \\
 &= \mu(X_y^{(n+1)}) \bigg( 1 - \bigg( 1 - \frac{1}{2C_d} \bigg)^{n+1} \bigg),
\end{align*}
where the penultimate step used \eqref{quotient_doubling}. This proves \eqref{eq:claimEN_induction}.
\\~\\
\textbf{Step 3.}
For $n\in \N$ we define
 \[ 
 \widetilde{Y}^{({n})} :=\{y\in Y:  \exists y_{n} \in Y_{n}, \ X_y \subseteq \widetilde X^{(n)}_{y_n} \} .
 \] 
In this step, we set $\mu_0 := \mu$ and construct measures $\mu_n$, $n \in \mathbb{N}$, on $X$ that are density equivalent to $\mu$ and satisfy
\begin{align} \label{eq:measure_property1}
D^-_{\mu} (\Lambda) - \varepsilon  & \leq \frac{\# (X_y \cap \Lambda)}{\mu_n (X_y)} 
 \leq D_{\mu}^+ (\Lambda) + \varepsilon \qquad\text{for all } y \in \widetilde Y^{(n)},
\\
\label{eqm}
 \mu_n(X_{y}^{(n)}) &= \mu_{n-1}(X_{y}^{(n)})
\qquad\text{for all }y \in Y_n.
\end{align}
Suppose that we have constructed measures $\mu_0, \ldots, \mu_{n-1}$ for some $n\in\N$, which are mutually density equivalent.  We apply Lemma \ref{step} for the measure $\mu_{n-1}$, the pair of partitions $\{X_y\}_{y\in Y}$ and $\{X^{(n-1)}_y\}_{y\in Y_{n-1}}$, and the family of balls $\{B_{R_n+6R_{n-1}} (y)\}_{y \in Y_n}$.
This implies the existence of a measure $\mu_n$, which is density equivalent to $\mu_{n-1}$, such that \eqref{eq:measure_property1} holds and
\begin{align}
\label{eq:measure_property2}
\mu_n (X_y) &= \mu_{n-1}(X_y) \qquad\text{for all }y \in Y  \setminus \widetilde Y^{(n)},
\\
\mu_n(\widetilde{X}_{y}^{(n)}) & =\mu_{n-1}(\widetilde{X}_{y}^{(n)})\qquad\text{ for all }y\in Y_n.
\end{align}
Hence, for any $y\in Y_n$, we have
\[
\begin{aligned}
\mu_n(X_{y}^{(n)}) & = 
 \sum_{z \in  \widetilde{Y}^{(n)} : \ X_z \subseteq X_y^{(n)}} \mu_n(X_z)
+ \sum_{z \in Y \setminus \widetilde{Y}^{(n)} : \ X_z \subseteq X_y^{(n)}} \mu_n(X_z)
\\
& =\mu_n(\widetilde{X}_{y}^{(n)}) + \sum_{z \in Y \setminus \widetilde{Y}^{(n)} : \ X_z \subseteq X_y^{(n)}} \mu_n(X_z)
\\
&=\mu_{n-1}(\widetilde{X}_{y}^{(n)}) + \sum_{z \in Y \setminus \widetilde{Y}^{(n)} : \ X_z \subseteq X_y^{(n)}} \mu_{n-1}(X_z)
= \mu_{n-1}(X_{y}^{(n)}),
\end{aligned}
\]
which proves \eqref{eqm}.  Applying iteratively \eqref{eqm} and  the nestedness property \eqref{nestn}  yields 
\begin{equation}
\label{emm}
 \mu_n(X_{y}^{(n)}) = \sum_{z\in Y_{n-1}: \ X_z^{(n-1)} \subseteq X_y^{(n)}} \mu_{n-1}(X_{z}^{(n-1)}) = \ldots
 = \mu_{0}(X_{y}^{(n)}).
\end{equation}
for $y \in Y_n$.
\\~\\
\textbf{Step 4.} Lastly, we construct a required partition $\{X'_{y}\}_{y \in Y'}$ of $X$. Choose $n \in \mathbb{N}$ such that
\[
\bigg(1 - \frac{1}{2 C_d} \bigg)^n \le \varepsilon.
\]
Let  $R'=R_n$ and $X'_y=X^{(n)}_y$ for $y\in Y':=Y_n$. By \eqref{xy2} we have \eqref{xyh}. The nestedness property \eqref{nest} is a a consequence of \eqref{nestn}.
Finally, to prove \eqref{reg0}, define the set 
\[
\widetilde Y=\widetilde Y^{(1)} \cup \ldots \cup \widetilde Y^{(n)} = \{y \in Y: X_y \subseteq E_n\}.
\]
We claim that
\begin{equation}\label{reg6} 
 D^-_{\mu} (\Lambda) -\ve   \leq \frac{\# (X_y \cap \Lambda)}{\mu_n (X_y)} 
 \leq D_{\mu}^+ (\Lambda) + \ve \qquad\text{for all } y \in  \widetilde Y.
\end{equation}
Indeed, suppose that $y\in \widetilde{Y}$. Let $1\le k\le n$ be the largest integer such that $y\in \widetilde Y^{(k)}$. Since $\mu_k(X_y)=\mu_{k+1}(X_y) = \ldots = \mu_n(X_y)$, the bound \eqref{eq:measure_property1} with $n$ replaced by $k$ yields \eqref{reg6}.

Finally, take any $y\in Y_n$. By \eqref{eq:claimEN_induction} we have
\begin{equation}\label{reg7}
 \frac{\mu(X_y^{(n)} \cap (X\setminus E_n ) )}{\mu(X_y^{(n)})} \le \bigg(1- \frac{1}{2 C_d } \bigg)^n \le \ve.
\end{equation}
Consequently, by \eqref{eq:stephyp2} and \eqref{reg6} 
\[
\begin{aligned}
\#(X_y^{(n)} \cap \Lambda) &\le \sum_{z \in \widetilde Y: \ X_z \subseteq X_y^{(n)}} \#(X_z \cap \Lambda) +  \sum_{z \in Y \setminus \widetilde Y : \ X_z \subseteq X_y^{(n)}} \#(X_z \cap \Lambda) \\
& \le \sum_{z \in \widetilde Y: \ X_z \subseteq X_y^{(n)}}  (D^+_{\mu}(\Lambda)+\ve)\mu_n(X_z)  +   \sum_{z \in Y \setminus \widetilde Y : \ X_z \subseteq X_y^{(n)}} C_0 D^+_{\mu}(\Lambda) \mu_{n}(X_z ) \\
& = (D^+_{\mu}(\Lambda)+\ve)\mu_n(X_y^{(n)} \cap E_n) + C_0 D^+_{\mu}(\Lambda) \mu_0(X_y^{(n)} \cap (X \setminus E_n))
\\
& \le (D^+_{\mu}(\Lambda)+\ve)\mu_0(X_y^{(n)}) + C_0 D^+_{\mu}(\Lambda) \ve \mu_0(X_y^{(n)})
\\
&  = (D^+_{\mu}(\Lambda)(1+ C_0 \ve) + \ve )\mu(X_y^{(n)}),
\end{aligned}
\]
where the penultimate step follows from \eqref{emm} and \eqref{reg7}. Likewise, we show that
\[
\begin{aligned}
\#(X_y^{(n)} \cap \Lambda) 
& \ge   \sum_{z \in \widetilde Y: \ X_z \subseteq X_y^{(n)}}  (D^-_{\mu}(\Lambda) - \ve)\mu_n(X_z)
\\
& =
 (D^-_{\mu}(\Lambda) - \ve) ( \mu_n(X_y^{(n)} ) - \mu_n(X_y^{(n)} \cap (X \setminus E_n)))
\\
& =
 (D^-_{\mu}(\Lambda) - \ve) ( \mu_0(X_y^{(n)} ) - \mu_0(X_y^{(n)} \cap (X \setminus E_n)))
 \\
 & \ge  (D^-_{\mu}(\Lambda) - \ve) (1 - \ve)  \mu_0(X_y^{(n)} ) .
 \end{aligned} 
\]
Since $\ve>0$ is arbitrary, this shows \eqref{reg0}.
\end{proof}

\section*{Acknowledgements}
The authors thank Felix Voigtlaender for a helpful discussion on settling the claim in the proof of Theorem \ref{thm:frd_exponential}.
The first author was partially supported by the NSF grant DMS-2349756.
For J.~v.~V., this research was funded in whole or in part by the Austrian
Science Fund (FWF): 10.55776/J4555 and 10.55776/PAT2545623.

\bibliographystyle{abbrv}
\bibliography{bib}

\end{document}